\documentclass[12pt]{article}
\usepackage{mycommands}
\usepackage{color}
\usepackage{tikz}
\title{The Generality of a Section of a Curve} 

\begin{document}
\maketitle

\begin{abstract}
Let $f \colon C \to \pp^3$ be a general curve of genus~$g$,
mapped to $\pp^3$ via a general linear series of degree~$d$;
and let $Q$ be a general (and thus smooth) quadric.
In this paper, we show that the points of intersection $f(C) \cap Q$
give a \emph{general} collection of $2d$ points on $Q$,
except for exactly six exceptional cases.  

We also prove similar theorems for every other pair $(r, n)$
for which, except for only finitely many pairs $(d, g)$,
the intersection of a general curve of genus~$g$ mapped to $\pp^r$
via a general linear series of degree~$d$, with a general hypersurface $S$
of degree $n$, is a general collection of $dn$ points on $S$.

As explained in \cite{over}, these results play a key role in the author's proof
of the Maximal Rank Conjecture \cite{mrc}.
\end{abstract}

\section{Introduction}

If $C$ is a general curve of genus $g$, equipped with a general map
$f \colon C \to \pp^3$ of degree $d$,
it is natural to ask
whether the intersection $f(C) \cap Q$
of its image with a general quadric $Q$
is a general collection of $2d$ points on $Q$.
Interest in this question historically developed as a result of the 
work of Hirschowitz \cite{mrat} on the Maximal Rank Conjecture
for rational space curves, and the later extension of Ballico
and Ellia \cite{ball} of this method to nonspecial space curves: The
heart of these arguments revolve precisely around understanding the intersection
of a general curve with a general quadric.
In hopes of both simplifying and extending these results,
Ellingsrud and Hirschowitz \cite{eh}, and later Perrin \cite{perrin},
using the technique of liaison,
gave partial results on the generality of this intersection.
However, a complete analysis has so far remained conjectural.

The present paper gives the first complete analysis.
Both the results and the techniques developed here play a critical role
in the author's proof of the Maximal Rank Conjecture \cite{mrc}, as
explained in \cite{over}.

\medskip

\noindent
To state the problem precisely, we make the following definition:

\begin{defi}
We say a stable map $f \colon C \to \pp^r$ from a curve $C$ to $\pp^r$
(with $r \geq 2$)
is a \emph{Weak Brill--Noether curve (WBN-curve)} if it
corresponds to a point in a component of
$\bar{M}_g(\pp^r, d)$ which both
dominates $\bar{M}_g$,
and whose generic member is a map
from a smooth curve, which is an immersion if $r \geq 3$,
and birational onto its image if $r = 2$;
and which is either
nonspecial or nondegenerate.
In the latter case, we refer to it as a \emph{Brill--Noether curve} (\emph{BN-curve}).
\end{defi}

\noindent
The celebrated Brill--Noether theorem
then asserts that BN-curves of degree~$d$ and genus~$g$ to~$\pp^r$ exist if and only if
\[\rho(d, g, r) := (r + 1)d - rg - r(r + 1) \geq 0.\]
Moreover, for $\rho(d, g, r) \geq 0$, the parameter space
of BN-curves is irreducible. (In particular, it makes sense
to talk about a ``general BN-curve''.)

\medskip

In this paper, we give a complete answer to the question posed above:
For $f \colon C \to \pp^3$
a general BN-curve of degree $d$ and genus $g$
(with, of course, $\rho(d, g, 3) \geq 0$),
we show the intersection $f(C) \cap Q$ is a general collection of $2d$ points on $Q$
except in exactly six cases. Furthermore, in these six cases, we compute precisely
what the intersection is.

A natural generalization of this problem is to study the intersection of
a general BN-curve $f \colon C \to \pp^r$ (for $r \geq 2$) with a hypersurface $H$
of degree $n \geq 1$: In particular, we ask when this intersection consists
of a general collection of $dn$ points on $H$ (in all but finitely many cases).

For $r = 2$, the divisor $f(C) \cap H$ on $H$ is linearly equivalent
to $\oo_H(d)$; in particular, it can only be general if $H$ is rational, i.e.\ if $n = 1$ or $n = 2$.
In general, we note that
in order for the intersection to be general, it is evidently necessary for
\[(r + 1)d - (r - 3)g \sim (r + 1)d - (r - 3)(g - 1) = \dim \bar{M}_g(\pp^r, d)^\circ \geq (r - 1) \cdot dn.\]
(Here $\bar{M}_g(\pp^r, d)^\circ$ denotes the component of $\bar{M}_g(\pp^r, d)$
corresponding to the BN-curves, and $A \sim B$ denotes that $A$ differs from $B$ by a quantity bounded by
a function of $r$ alone.)
If the genus of $C$ is as large as possible (subject to the constraint
that $\rho(d, g, r) \geq 0$), i.e.\ if
\[g \sim \frac{r + 1}{r} \cdot d,\]
then the intersection can only be general when
\[(r + 1) \cdot d - (r - 3) \cdot \left(\frac{r + 1}{r} \cdot d \right) \gtrsim (r - 1) n \cdot d;\]
or equivalently if
\[(r + 1) - (r - 3) \cdot \frac{r + 1}{r} \geq (r - 1) n \quad \Leftrightarrow \quad n \leq \frac{3r + 3}{r^2 - r}.\]

For $r = 3$, this implies $n = 1$ or $n = 2$; for $r = 4$, this implies $n = 1$; and
for $r \geq 5$, this is impossible.

\medskip

To summarize, there are only
five pairs $(r, n)$ where this intersection could be, with the exception of finitely many
$(d, g)$ pairs,
a collection of $dn$ general points on $H$: The intersection of a plane curve with a line,
the intersection of a plane curve with a conic, the intersection of a space curve with a quadric,
the intersection of a space curve with a plane, and the intersection of a curve to $\pp^4$
with a hyperplane.
Our three main theorems (five counting the first two cases which are trivial)
give a complete description of this intersection
in these cases:

\begin{thm} \label{main-2}
Let $f \colon C \to \pp^2$ be a general BN-curve of degree~$d$ and genus~$g$. Then
the intersection $f(C) \cap Q$, of $C$ with a general conic $Q$, consists
of a general collection of $2d$ points on~$Q$.
\end{thm}

\begin{thm} \label{main-2-1}
Let $f \colon C \to \pp^2$ be a general BN-curve of degree~$d$ and genus~$g$. Then
the intersection $f(C) \cap L$, of $C$ with a general line $L$, consists
of a general collection of $d$ points on~$L$.
\end{thm}

\begin{thm} \label{main-3}
Let $f \colon C \to \pp^3$ be a general BN-curve of degree~$d$ and genus~$g$. Then
the intersection $f(C) \cap Q$, of $C$ with a general quadric $Q$, consists
of a general collection of $2d$ points on $Q$, unless
\[(d, g) \in \{(4, 1), (5, 2), (6, 2), (6, 4), (7, 5), (8, 6)\}.\]
And conversely, in the above cases, we may describe the intersection
$f(C) \cap Q \subset Q \simeq \pp^1 \times \pp^1$ in terms of
the intrinsic geometry of $Q \simeq \pp^1 \times \pp^1$ as follows:

\begin{itemize}
\item If $(d, g) = (4, 1)$, then $f(C) \cap Q$ is the intersection of two general curves
of bidegree $(2, 2)$.

\item If $(d, g) = (5, 2)$, then $f(C) \cap Q$ is a general collection of $10$ points
on a curve of bidegree~$(2, 2)$.

\item If $(d, g) = (6, 2)$, then $f(C) \cap Q$ is a general collection of $12$ points
$p_1, \ldots, p_{12}$ lying on a curve $D$ which satisfy:
\begin{itemize}
\item The curve $D$ is of bidegree $(3, 3)$ (and so is in particular of arithmetic genus $4$).
\item The curve $D$ has two nodes (and so is in particular of geometric genus $2$).
\item The divisors $\oo_D(2,2)$ and $p_1 + \cdots + p_{12}$ are linearly equivalent
when pulled back to the normalization of $D$.
\end{itemize}

\item If $(d, g) = (6, 4)$, then $f(C) \cap Q$ is the intersection of
two general curves
of bidegrees $(2, 2)$ and $(3,3)$ respectively.

\item If $(d, g) = (7, 5)$, then $f(C) \cap Q$ is a general collection of $14$ points
$p_1, \ldots, p_{14}$ lying on a curve $D$ which satisfy:
\begin{itemize}
\item The curve $D$ is of bidegree $(3, 3)$.
\item The divisor $p_1 + \cdots + p_{14} - \oo_D(2, 2)$ on $D$
is effective.
\end{itemize}

\item If $(d, g) = (8, 6)$, then $f(C) \cap Q$ is a general collection of $16$
points on a curve of bidegree~$(3,3)$.
\end{itemize}
In particular, the above descriptions show $f(C) \cap Q$ is not a general collection
of $2d$ points on~$Q$.
\end{thm}

\begin{thm} \label{main-3-1}
Let $f \colon C \to \pp^3$ be a general BN-curve of degree~$d$ and genus~$g$. Then
the intersection $f(C) \cap H$, of $C$ with a general plane $H$, consists
of a general collection of $d$ points on $H$, unless
\[(d, g) = (6, 4).\]
And conversely, for $(d, g) = (6, 4)$, the intersection $f(C) \cap H$
is a general collection of $6$ points on a conic in $H \simeq \pp^2$; in particular,
it is not a general collection of $d = 6$ points.
\end{thm}

\begin{thm} \label{main-4}
Let $f \colon C \to \pp^4$ be a general BN-curve of degree~$d$ and genus~$g$. Then
the intersection $f(C) \cap H$, of $C$ with a general hyperplane $H$, consists
of a general collection of $d$ points on $H$, unless
\[(d, g) \in \{(8, 5), (9, 6), (10, 7)\}.\]
And conversely, in the above cases, we may describe the intersection
$f(C) \cap H \subset H \simeq \pp^3$ in terms of
the intrinsic geometry of $H \simeq \pp^3$ as follows:

\begin{itemize}
\item If $(d, g) = (8, 5)$, then $f(C) \cap H$ is the intersection of three general quadrics.
\item If $(d, g) = (9, 6)$, then $f(C) \cap H$ is a general collection of $9$ points
on a curve $E \subset \pp^3$ of degree~$4$ and genus~$1$.

\item If $(d, g) = (8, 5)$, then $f(C) \cap H$ is a general collection of $10$ points
on a quadric.
\end{itemize}
\end{thm}

The above theorems can be proven by studying the normal bundle of
the general BN-curve $f \colon C \to \pp^r$: For any hypersurface $S$ of degree $n$,
and unramified map $f \colon C \to \pp^r$ dimensionally transverse to $S$,
basic deformation theory implies that the map
\[f \mapsto (f(C) \cap S)\]
(from the corresponding Kontsevich space of stable maps, to the
corresponding symmetric power of $S$)
is smooth at $[f]$ if and only if
\[H^1(N_f(-n)) = 0.\]
Here, $N_f(-n) = N_f \otimes f^* \oo_{\pp^r}(-n)$
denotes the twist of the normal bundle $N_f$ of the map $f \colon C \to \pp^r$;
this is the vector bundle on the domain $C$ of $f$ defined via
\[N_f = \ker(f^* \Omega_{\pp^r} \to \Omega_C)^\vee.\]

Since a map between reduced irreducible varieties is dominant
if and only if it is generically smooth, the map $f \mapsto (f(C) \cap S)$ is therefore dominant if and only if
$H^1(N_f(-n)) = 0$ for $[f]$ general.

This last condition being visibly open, our problem is thus to prove
the existence of an unramified BN-curve $f \colon C \to \pp^r$ of specified degree and genus,
for which $H^1(N_f(-n)) = 0$.
For this, we will use a variety of techniques, most crucially specialization
to a map from a reducible curve $X \cup_\Gamma Y \to \pp^r$.

We begin, in
Section~\ref{sec:reducible}, by giving several tools
for studying the normal bundle of a map from a reducible curve.
Then in
Section~\ref{sec:inter}, we review results on the closely-related
\emph{interpolation problem} (c.f.\ \cite{firstpaper}).
In Section~\ref{sec:rbn}, we review results about when certain maps from reducible
curves, of the type we shall use, are BN-curves.
Using these techniques, we then concentrate our attention in Section~\ref{sec:indarg} on
maps from reducible curves $X \cup_\Gamma Y \to \pp^r$ where $Y$ is a line or canonical curve.
Consideration of these curves enables us to make an inductive argument
that reduces our main theorems to finite casework.

This finite casework is then taken care of in three steps:
First, in Sections~\ref{sec:hir}--\ref{sec:hir-3}, we again use degeneration
to a map from a reducible curve, considering the special case when $Y \to \pp^r$ factors through a
hyperplane.
Second, in Section~\ref{sec:in-surfaces},
we specialize to immersions of smooth curves contained in Del Pezzo surfaces, and study
the normal bundle of our curve using the
normal bundle exact sequence for a curve in a surface.
Lastly, in Section~\ref{sec:51} we use the geometry of the cubic scroll in $\pp^4$ to
construct an example of an immersion of a smooth curve $f \colon C \hookrightarrow \pp^3$ of degree $5$
and genus $1$ with $H^1(N_f(-2)) = 0$.

Finally, in Section~\ref{sec:converses}, we examine each of the cases
in our above theorems where the intersection is not general. In each
of these cases, we work out precisely what the intersection is
(and show that it is not general).

\subsection*{Conventions}

In this paper we make the following conventions:

\begin{itemize}
\item We work over an algebraically closed field of characteristic zero.

\item
A \emph{curve} shall refer to a nodal curve, which is assumed to be connected unless otherwise specified.
\end{itemize}

\subsection*{Acknowledgements}

The author would like to thank Joe Harris for
his guidance throughout this research.
The author would also like to thank Gavril Farkas, Isabel Vogt, and
members of the Harvard and MIT mathematics departments
for helpful conversations;
and to acknowledge the generous
support both of the Fannie and John Hertz Foundation,
and of the Department of Defense
(NDSEG fellowship).

\section{Normal Bundles of Maps from Reducible Curves \label{sec:reducible}}

In order to describe the normal bundle of a map from a reducible curve,
it will be helpful to introduce some notions concerning modifications
of vector bundles.
The interested reader is encouraged to consult \cite{firstpaper} (sections 2, 3, and~5),
where these notions are developed in full; we include here only a brief summary, which will
suffice for our purposes.

\begin{defi}
If $f \colon X \to \pp^r$ is a map from a scheme $X$ to $\pp^r$,
and $p \in X$ is a point, we write $[T_p C] \subset \pp^r$
for the \emph{projective realization of the tangent space} --- i.e.\ for the
linear subspace $L \subset \pp^r$ containing $f(p)$ and satisfying
$T_{f(p)} L = f_*(T_p C)$.
\end{defi}

\begin{defi} Let $\Lambda \subset \pp^r$ be a linear subspace, and $f \colon C \to \pp^r$
be an unramified map from a curve.
Write $U_{f, \Lambda} \subset C$ for the open subset of points $p \in C$ so that
the projective realization of the tangent space $[T_p C]$ does not meet $\Lambda$. Suppose that $U_{f, \Lambda}$
is nonempty, and contains the singular locus of $C$. Define
\[N_{f \to \Lambda}|_{U_{f, \Lambda}} \subset N_f|_{U_{f, \Lambda}}\]
as the kernel of the differential of the projection from $\Lambda$
(which is regular on a neighborhood of $f(U_{f, \Lambda})$).
We then let $N_{f \to \Lambda}$ be the unique extension of $N_{f \to \Lambda}|_{U_{f, \Lambda}}$
to a sub-vector-bundle (i.e.\ a subsheaf with locally free quotient) of $N_f$ on $C$.
For a more thorough discussion of this construction (written for $f$ an immersion
but which readily generalizes),
see Section~5 of \cite{firstpaper}.
\end{defi}

\begin{defi} Given a subbundle $\mathcal{F} \subset \mathcal{E}$ of a vector bundle on a scheme $X$,
and a Cartier divisor $D$ on $X$, we define
\[\mathcal{E}[D \to \mathcal{F}]\]
as the kernel of the natural map
\[\mathcal{E} \to (\mathcal{E} / \mathcal{F})|_D.\]
Note that $\mathcal{E}[D \to \mathcal{F}]$ is naturally isomorphic to $\mathcal{E}$
on $X \smallsetminus D$. Additionally, note that $\mathcal{E}[D \to \mathcal{F}]$
depends only on $\mathcal{F}|_D$.
For a more thorough discussion of this construction, see Sections~2 and~3 of \cite{firstpaper}.
\end{defi}

\begin{defi}
Given a subspace $\Lambda \subset \pp^r$, an unramified map $f \colon C \to \pp^r$ from a curve, and a Cartier divisor $D$ on $C$,
we define
\[N_f[D \to \Lambda] := N_f[D \to N_{f \to \Lambda}].\]
\end{defi}

We note that these constructions can be iterated on a smooth curve: Given subbundles $\mathcal{F}_1, \mathcal{F}_2 \subset \mathcal{E}$
of a vector bundle on a smooth curve,
there is a unique subbundle $\mathcal{F}_2' \subset \mathcal{E}[D_1 \to \mathcal{F}_1]$
which agrees with $\mathcal{F}_2$ away from $D_1$ (c.f.\ Proposition~3.1 of \cite{firstpaper}).
We may then define:
\[\mathcal{E}[D_1 \to \mathcal{F}_1][D_2 \to \mathcal{F}_2] := \mathcal{E}[D_1 \to \mathcal{F}_1][D_2 \to \mathcal{F}_2'].\]

Basic properties of this construction (as well as precise conditions when such iterated modifications
make sense for higher-dimensional
varieties) are investigated in \cite{firstpaper} (Sections~2 and~3).
For example, we have natural isomorphisms $\mathcal{E}[D_1 \to \mathcal{F}_1][D_2 \to \mathcal{F}_2] \simeq \mathcal{E}[D_2 \to \mathcal{F}_2][D_1 \to \mathcal{F}_1]$
in several cases, including when $\mathcal{F}_1 \subseteq \mathcal{F}_2$.

Using these constructions, we may give a partial characterization of the
normal bundle $N_f$ of an unramified map from a reducible curve $f \colon X \cup_\Gamma Y \to \pp^r$:

\begin{prop}[Hartshorne-Hirschowitz]
Let $f \colon X \cup_\Gamma Y \to \pp^r$ be an unramified map from a reducible curve.
Write $\Gamma = \{p_1, p_2, \ldots, p_n\}$,
and for each $i$ let $q_i \neq f(p_i)$ be a point on the projective realization
$[T_{p_i} Y]$ of the tangent space to $Y$ at $p_i$. Then we have
\[N_f|_X = N_{f|_X}(\Gamma)[p_1 \to q_1][p_2 \to q_2] \cdots [p_n \to q_n].\]
\end{prop}
\begin{proof}
This is Corollary~3.2 of \cite{hh}, re-expressed in the above
language. (Hartshorne and Hirschowitz states this only for $r = 3$
and $f$ an immersion; but the argument they give works for $r$ arbitrary.)
\end{proof}

Our basic strategy to study the normal bundle of an unramified map from a
reducible curve $f \colon C \cup_\Gamma D \to \pp^r$
is given by the following lemma:

\begin{lm} \label{glue}
Let $f \colon C \cup_\Gamma D \to \pp^r$ be an unramified map from a reducible curve,
and let $E$ and $F$ be
divisors supported on $C \smallsetminus \Gamma$ and $D \smallsetminus \Gamma$
respectively.
Suppose that the natural map
\[\alpha \colon H^0(N_{f|_D}(-F)) \to \bigoplus_{p \in \Gamma} \left(\frac{T_p (\pp^r)}{f_* (T_p (C \cup_\Gamma D))}\right)\]
is surjective (respectively injective), and that
\begin{gather*}
H^1(N_f|_D (-F)) = 0 \quad \text{(respectively } H^0(N_f|_D (-F)) = H^0(N_{f|_D} (-F))\text{)} \\
H^1(N_{f|_C} (-E)) = 0 \quad \text{(respectively } H^0(N_{f|_C} (-E)) = 0\text{)}.
\end{gather*}
Then we have
\[H^1(N_f(-E-F)) = 0  \quad \text{(respectively } H^0(N_f(-E-F)) = 0\text{)}.\]
\end{lm}
\begin{proof}
Write $\mathcal{K}$ for the sheaf supported along $\Gamma$ whose
stalk at $p \in \Gamma$ is the quotient of tangent spaces:
\[\mathcal{K}_p = \frac{T_p(\pp^r)}{f_*(T_p(C \cup_\Gamma D))}.\]
Additionally, write $\mathcal{N}$ for the (not locally-free) subsheaf of $N_f$
``corresponding to deformations which do not smooth the nodes $\Gamma$''; or in
symbols, as the kernel of the natural map
\[N_f \to T^1_\Gamma,\]
where $T^1$ is the Lichtenbaum-Schlessinger $T^1$-functor.
We have the following exact sequences of sheaves:
\[\begin{CD}
0 @>>> \mathcal{N} @>>> N_f @>>> T^1_\Gamma @>>> 0 \\
@. @VVV @VVV @| @. \\
0 @>>> N_{f|_D} @>>> N_f|_D @>>> T^1_\Gamma @>>> 0 \\
@. @. @. @. @. \\
0 @>>> \mathcal{N} @>>> N_{f|_C} \oplus N_{f|_D} @>>> \mathcal{K} @>>> 0. \\
\end{CD}\]

The first of sequence above is just the definition of $\mathcal{N}$.
Restriction of the first sequence to~$D$ yields the second sequence
(we have $\mathcal{N}|_D \simeq N_{f|_D}$);
the map between them being of course the restriction map.
The final sequence expresses $\mathcal{N}$ as the gluing of $\mathcal{N}|_C \simeq N_{f|_C}$
to $\mathcal{N}|_D \simeq N_{f|_D}$ along $\mathcal{N}|_\Gamma \simeq \mathcal{K}$.

Twisting everything in sight by $-E-F$, we obtain new sequences:
\[\begin{CD}
0 @>>> \mathcal{N}(-E-F) @>>> N_f(-E-F) @>>> T^1_\Gamma @>>> 0 \\
@. @VVV @VVV @| @. \\
0 @>>> N_{f|_D}(-F) @>>> N_f|_D(-F) @>>> T^1_\Gamma @>>> 0 \\
@. @. @. @. @. \\
0 @>>> \mathcal{N}(-E-F) @>>> N_{f|_C}(-E) \oplus N_{f|_D}(-F) @>>> \mathcal{K} @>>> 0. \\
\end{CD}\]

The commutativity of the rightmost square in the first diagram implies that
the image of $H^0(N_f(-E-F)) \to H^0(T^1_\Gamma)$
is contained in the image of $H^0(N_f|_D(-F)) \to H^0(T^1_\Gamma)$.
Consequently, we have
\begin{align}
\dim H^0(N_f(-E-F)) &= \dim H^0(\mathcal{N}(-E-F)) + \dim \im\left(H^0(N_f(-E-F)) \to H^0(T^1_\Gamma)\right) \nonumber \\
&\leq \dim H^0(\mathcal{N}(-E-F)) + \dim \im\left(H^0(N_f|_D(-F)) \to H^0(T^1_\Gamma)\right) \nonumber \\
&= \dim H^0(\mathcal{N}(-E-F)) + \dim H^0(N_f|_D(-F)) - \dim H^0(N_{f|_D}(-F)). \label{glue-dim}
\end{align}

Next, our assumption that $H^0(N_{f|_D}(-F)) \to H^0(\mathcal{K})$ is surjective
(respectively our assumptions that $H^0(N_{f|_C}(-E)) = 0$ and $H^0(N_{f|_D}(-F)) \to H^0(\mathcal{K})$ is injective) implies
in particular that $H^0(N_{f|_C}(-E) \oplus N_{f|_D}(-F)) \to H^0(\mathcal{K})$ is surjective (respectively injective).

In the ``respectively'' case, this yields $H^0(\mathcal{N}(-E-F)) = 0$, which combined with \eqref{glue-dim}
and our assumption that $H^0(N_f|_D(-F)) = H^0(N_{f|_D}(-F))$ implies $H^0(N_f(-E-F)) = 0$ as desired.
In the other case, we have a bit more work to do; the surjectivity of
$H^0(N_{f|_D}(-F)) \to H^0(\mathcal{K})$ yields
\[\dim H^0(\mathcal{N}(-E-F)) = \dim H^0(N_{f|_C}(-E) \oplus N_{f|_D}(-F)) - \dim H^0(\mathcal{K});\]
or upon rearrangement,
\begin{align*}
\dim H^0(\mathcal{N}(-E-F)) - \dim H^0(N_{f|_D}(-F)) &= \dim H^0(N_{f|_C}(-E)) - \dim H^0(\mathcal{K}) \\
&= \chi(N_{f|_C}(-E)) - \chi(\mathcal{K}).
\end{align*}
(For the last equality, $\dim H^0(N_{f|_C}(-E)) = \chi(N_{f|_C}(-E)) + \dim H^1(N_{f|_C}(-E)) = \chi(N_{f|_C}(-E))$
because $H^1(N_{f|_C}(-E)) = 0$ by assumption. Additionally, 
$\dim H^0(\mathcal{K}) = \chi(\mathcal{K})$
because $\mathcal{K}$ is punctual.)

Substituting this into \eqref{glue-dim}, and noting that
$\dim H^0(N_f|_D(-F)) = \chi(N_f|_D(-F))$ because
$H^1(N_f|_D(-F)) = 0$ by assumption, we obtain:
\begin{align}
\dim H^0(N_f(-E-F)) &\leq \dim H^0(N_f|_D(-F)) + \dim H^0(\mathcal{N}(-E-F)) - \dim H^0(N_{f|_D}(-F)) \nonumber \\
&= \chi(N_f|_D(-F)) + \chi(N_{f|_C}(-E)) - \chi(\mathcal{K}) \nonumber \\
&= \chi(N_f|_D(-F)) + \chi(N_f|_C(-E - \Gamma)) \nonumber \\
&= \chi(N_f(-E - F)). \label{glue-done}
\end{align}
For the final two equalities, we have used the exact sequences of sheaves
\begin{gather*}
0 \to N_f|_C(-E - \Gamma) \to N_{f|_C}(-E) \to \mathcal{K} \to 0 \\[1ex]
0 \to N_f|_C(-E - \Gamma) \to N_f(-E-F) \to N_f|_D(-F) \to 0;
\end{gather*}
which are just twists by $-E-F$ of the exact sequences:
\begin{gather*}
0 \to N_f|_C(-\Gamma) \to N_{f|_C} \to \mathcal{K} \to 0 \\[1ex]
0 \to N_f|_C(-\Gamma) \to N_f \to N_f|_D \to 0.
\end{gather*}

\noindent
To finish, we note that, by \eqref{glue-done},
\[\dim H^1(N_f(-E-F)) = \dim H^0(N_f(-E-F)) - \chi(N_f(-E - F)) \leq 0,\]
and so
$H^1(N_f(-E-F)) = 0$ as desired.
\end{proof}

In the case where $f|_D$ factors through a hyperplane,
the hypotheses of Lemma~\ref{glue} become easier to check:

\begin{lm} \label{hyp-glue}
Let $f \colon C \cup_\Gamma D \to \pp^r$ be an unramified map from a reducible curve,
such that $f|_D$ factors as a composition of $f_D \colon D \to H$ with the inclusion of a hyperplane $\iota \colon H \subset \pp^r$,
while $f|_C$ is transverse to $H$ along $\Gamma$.
Let $E$ and $F$ be
divisors supported on $C \smallsetminus \Gamma$ and $D \smallsetminus \Gamma$
respectively.
Suppose that, for some $i \in \{0, 1\}$,
\[H^i(N_{f_D}(-\Gamma-F)) = H^i(\oo_D(1)(\Gamma-F)) = H^i(N_{f|_C} (-E)) = 0.\]
Then we have
\[H^i(N_f(-E-F)) = 0.\]
\end{lm}
\begin{proof}
If $i = 0$, we note that $H^0(\oo_D(1)(\Gamma - F)) = 0$ implies
$H^0(\oo_D(1)(-F)) = 0$. In particular, using
the exact sequences
\[\begin{CD}
0 @>>> N_{f_D}(-F) @>>> N_{f|_D}(-F) @>>> \oo_D(1)(-F) @>>> 0 \\
@. @| @VVV @VVV @. \\
0 @>>> N_{f_D}(-F) @>>> N_f|_D(-F) @>>> \oo_D(1)(\Gamma - F) @>>> 0,
\end{CD}\]
we conclude from the first sequence that
$H^0(N_{f_D}(-F)) \to H^0(N_{f|_D}(-F))$ is an isomorphism, and
from the $5$-lemma applied to the corresponding map
between long exact sequences that $H^0(N_{f|_D}(-F)) = H^0(N_f|_D(-F))$.

Similarly, when $i = 1$, we note that
$H^1(N_{f_D}(-\Gamma-F)) = 0$ implies $H^1(N_{f_D}(-F)) = 0$;
we thus conclude from the second sequence that $H^1(N_f|_D(-F)) = 0$.

It thus remains to check that the map $\alpha$ in Lemma~\ref{glue}
is injective if $i = 0$ and surjective if $i = 1$. For this we use
the commutative diagram
\[\begin{CD}
\displaystyle H^0(N_{f_D}(-F)) @>\beta>> N_{f_D}|_\Gamma \simeq \displaystyle \bigoplus_{p \in \Gamma} \left(\frac{T_p H}{f_*(T_p D)}\right) \\
@VgVV @VV{\iota_*}V \\
\displaystyle H^0(N_{f|_D}(-F)) @>\alpha>> \displaystyle \bigoplus_{p \in \Gamma} \left(\frac{T_p (\pp^r)}{f_*(T_p (C \cup_\Gamma D))}\right).
\end{CD}\]
Since $f|_C$ is transverse to $H$ along $\Gamma$, the
map $\iota_*$ above is an isomorphism. In particular,
since $g$ is an isomorphism when $i = 0$, it suffices to check
that $\beta$ is injective if $i = 0$ and surjective if $i = 1$.
But using the exact sequence
\[0 \to N_{f_D}(-\Gamma-F) \to N_{f_D}(-F) \to N_{f_D}|_\Gamma \to 0,\]
this follows from our assumption that $H^i(N_{f_D}(-\Gamma-F)) = 0$.
\end{proof}

\section{Interpolation \label{sec:inter}}

If we generalize $N_f(-n)$ to $N_f(-D)$, where $D$ is a general effective divisor,
we get the problem of ``interpolation.'' Geometrically, this corresponds to
asking if there is a curve of degree $d$ and genus $g$ which passes through a
collection of points which are general in $\pp^r$
(as opposed to general in a hypersurface $S$).
This condition is analogous in some sense to the conditions
of semistability and section-semistability
(see Section~3 of~\cite{nasko}), as well as to the 
Raynaud condition (property $\star$ of \cite{raynaud});
although we shall not make use of these analogies here.

\begin{defi} \label{def:inter} We say a vector bundle $\mathcal{E}$ on a curve $C$ \emph{satisfies interpolation}
if it is nonspecial, and for a general effective divisor $D$ of any degree,
\[H^0(\mathcal{E}(-D)) = 0 \tor H^1(\mathcal{E}(-D)) = 0.\]
\end{defi}

We have the following results on interpolation from \cite{firstpaper}.
To rephrase them in our current language,
note that if $f \colon C \to \pp^r$ is a general BN-curve for $r \geq 3$, then $f$ is an immersion,
so $N_f$ coincides with the normal bundle $N_{f(C)/\pp^r}$ of the image.
Note also that, from Brill--Noether theory,
a general BN-curve $f \colon C \to \pp^r$ of degree $d$ and genus $g$
is nonspecial (i.e.\ satisfies $H^1(f^* \oo_{\pp^r}(1)) = 0$) if and only if
$d \geq g + r$.

\begin{prop}[Theorem~1.3 of~\cite{firstpaper}] \label{inter}
Let $f \colon C \to \pp^r$ (for $r \geq 3$) be a general BN-curve of degree $d$ and genus $g$, where
\[d \geq g + r.\]
Then $N_f$ satisfies interpolation, unless
\[(d, g,r) \in \{(5,2,3), (6,2,4), (7,2,5)\}.\]
\end{prop}

\begin{prop}[Proposition~4.12 of~\cite{firstpaper}] \label{twist}
Let $\mathcal{E}$ be a vector bundle on a curve $C$, and $D$ be a divisor on $C$.
If $\mathcal{E}$ satisfies interpolation and
\[\chi(\mathcal{E}(-D)) \geq (\rk \mathcal{E}) \cdot (\operatorname{genus} C),\]
then $\mathcal{E}(-D)$ satisfies interpolation. In particular,
\[H^1(\mathcal{E}(-D)) = 0.\]
\end{prop}

\begin{lm} \label{g2} Let $f \colon C \to \pp^r$ (for $r \in \{3, 4, 5\}$)
be a general BN-curve of degree $r + 2$ and genus $2$.
Then $H^1(N_f(-1)) = 0$.
\end{lm}
\begin{proof}
We will show that there exists
an immersion $C \hookrightarrow \pp^r$, which is a BN-curve of degree $r + 2$ and genus $2$, and whose image
meets a hyperplane $H$ transversely
in a general collection of $r + 2$ points. For this, we first find a rational normal
curve $R \subset H$ passing through $r + 2$ general points, which is possible
by Corollary~1.4 of~\cite{firstpaper}.
This rational
normal curve is then the hyperplane section of some rational surface scroll $S \subset \pp^r$
(and we can freely choose the projective equivalence class of $S$).

It thus suffices to prove that there exists a smooth curve $C \subset S$,
for which $C \subset S \subset \pp^r$ is a BN-curve of degree $r + 2$ and genus $2$,
such that $C \cap (H \cap S)$ a set of $r + 2$ general points on $H \cap S$;
or alternatively such that the map
\[C \mapsto (C \cap (H \cap S)),\]
from the Hilbert scheme of curves on $S$, to the Hilbert scheme of points
on $H \cap S$,
is smooth at $[C]$; this in turn would follow from
$H^1(N_{C/S}(-1)) = 0$.

But by Corollary~13.3 of \cite{firstpaper}, the general BN-curve $C' \subset \pp^r$
(which is an immersion since $r \geq 3$) of degree $r + 2$ and genus $2$
in $\pp^r$ is contained in some rational surface
scroll $S'$, and satisfies $\chi(N_{C'/S'}) = 11$. Since we can choose $S$ projectively
equivalent to $S'$,
we may thus find a BN-curve $C \subset S$ of degree~$r + 2$
and genus~$2$ with $\chi(N_{C/S}) = 11$. But then,
\[\chi(N_{C/S}(-1)) = 11 - d \geq g \quad \Rightarrow \quad H^1(N_{C/S}(-1)) = 0. \qedhere\]
\end{proof}

\noindent
Combining these results, we obtain:

\begin{lm} \label{from-inter} Let $f \colon C \to \pp^r$ (for $r \geq 3$)
be a general BN-curve of degree $d$ and genus $g$.
Suppose that $d \geq g + r$.
\begin{itemize}
\item If $r = 3$ and $g = 0$, then $H^1(N_f(-2)) = 0$. In fact, $N_f(-2)$ satisfies interpolation.
\item If $r = 3$, then $H^1(N_f(-1)) = 0$. In fact, $N_f(-1)$ satisfies interpolation
except when $(d, g) = (5, 2)$.
\item If $r = 4$ and $d \geq 2g$, then $H^1(N_f(-1)) = 0$. In fact, $N_f(-1)$ satisfies interpolation
except when $(d, g) = (6, 2)$.
\end{itemize}
\end{lm}
\begin{proof}
When $(d, g, r) \in \{(5, 2, 3), (6, 2, 4)\}$, the desired result follows from Lemma~\ref{g2}.
Otherwise,
from Propositions~\ref{inter}, we know that $N_f$ satisfies interpolation.
Hence, the desired conclusion follows by applying
Proposition~\ref{twist}: If $r = 3$, then
\begin{align*}
\chi(N_f(-1)) &= 2d \geq 2g = (r - 1) g\\
\chi(N_f(-2)) &= 0 = (r - 1)g;
\end{align*}
and if $r = 4$ and $d \geq 2g$, then
\[\chi(N_f(-1)) = 2d - g + 1 \geq 3g = (r - 1)g. \qedhere \]
\end{proof}

\begin{lm} \label{addone-raw}
Suppose $f \colon C \cup_u L \to \pp^3$ is an unramified map
from a reducible curve, with $L \simeq \pp^1$, and $u$ a single point,
and $f|_L$ of degree~$1$.
Write $v \neq f(u)$ for some other point on $f(L)$. If
\[H^1(N_{f|_C}(-2)(u)[2u \to v]) = 0,\]
then we have
\[H^1(N_f(-2)) = 0.\]
\end{lm}
\begin{proof}
We apply Lemma~8.5 of \cite{firstpaper} (which is stated for $f$ an immersion,
in which case $N_f = N_{C \cup L}$ and $N_{f|_C} = N_C$, but the same proof works
whenever $f$ is unramified); we take $N_C' = N_{f|_C}(-2)$
and $\Lambda_1 = \Lambda_2 = \emptyset$. This implies $N_f(-2)$ satisfies
interpolation (c.f.\ Definition~\ref{def:inter}) provided that $N_{f|_C}(-2)(u)[u \to v][u \to v]$ satisfies interpolation.
But we have
\[\chi(N_f(-2)) = \chi(N_{f|_C}(-2)(u)[u \to v][u \to v]) = 0;\]
so both of these interpolation statements are equivalent to the vanishing of $H^1$.
That is, we have $H^1(N_f(-2)) = 0$, provided that
\[H^1(N_{f|_C}(-2)(u)[u \to v][u \to v]) = H^1(N_{f|_C}(-2)(u)[2u \to v]) = 0,\]
as desired.
\end{proof}

We finish this section with the following proposition,
which immediately implies Theorems~\ref{main-2} and~\ref{main-2-1}:

\begin{prop} \label{p2}
Let $f \colon C \to \pp^2$ be a curve. Then $N_f(-2)$ satisfies interpolation.
In particular $H^1(N_f(-2)) = H^1(N_f(-1)) = 0$.
\end{prop}
\begin{proof}
By adjunction,
\[N_f \simeq K_C \otimes f^* K_{\pp^3}^{-1} \simeq K_f(3) \imp N_f(-2) \simeq K_C(1).\]
By Serre duality,
\[H^1(K_C(1)) \simeq H^0(\oo_C(-1))^\vee = 0;\]
which since $K_C(1)$ is a line bundle implies it satisfies interpolation.
\end{proof}

\section{Reducible BN-Curves \label{sec:rbn}}

\begin{defi} Let $\Gamma \subset \pp^r$ be a finite set of $n$ points. A pair
$(f \colon C \to \pp^r, \Delta \subset C_{\text{sm}})$,
where $C$ is a curve, $f$ is map from $C$ to $\pp^r$, and $\Delta$ is a subset of $n$ points on the smooth locus $C_{\text{sm}}$,
shall be called a \emph{marked curve (respectively marked BN-curve, respectively marked WBN-curve) passing through $\Gamma$}
if $f \colon C \to \pp^r$ is a map from a curve (respectively a BN-curve, respectively a WBN-curve) and $f(\Delta) = \Gamma$.

Given a marked curve $(f \colon C \to \pp^r, \Delta)$ passing through $\Gamma$,
we realize $\Gamma$ as a subset of $C$ via
$\Gamma \simeq \Delta \subset C$.

For $p \in \Gamma$,
we then define the \emph{tangent line $T_p (f, \Gamma)$ at $p$} to be the unique line $\ell \subset \pp^r$ through $p$
with $T_p \ell = f_* T_p C$.
\end{defi}

Let $\Gamma \subset \pp^r$ be a finite set of $n$ general points,
and $(f_i \colon C_i \to \pp^r, \Gamma_i)$ be marked WBN-curves passing through $\Gamma$.
We then write $C_1 \cup_\Gamma C_2$ for the curve obtained
from $C_1$ and $C_2$ by gluing 
$\Gamma_1$ to $\Gamma_2$ via the isomorphism $\Gamma_1 \simeq \Gamma \simeq \Gamma_2$.
The maps $f_i$ give rise to a map $f \colon C_1 \cup_\Gamma C_2 \to \pp^r$
from a reducible curve.
Then we have the following result:

\begin{prop}[Theorem~1.3 of \cite{rbn}] \label{prop:glue}
Suppose that, for at least one $i \in \{1, 2\}$, we have
\[(r + 1) d_i - r g_i + r \geq rn.\]
Then
$f \colon C_1 \cup_\Gamma C_2 \to \pp^r$ is a WBN-curve.
\end{prop}

\begin{prop} \label{prop:interior}
In Proposition~\ref{prop:glue}, suppose that $[f_1, \Gamma_1]$ is general in some component
of the space of marked WBN-curves passing through $\Gamma$,
and that $H^1(N_{f_2}) = 0$. Then $H^1(N_f) = 0$.
\end{prop}
\begin{proof}
This follows from combining Lemmas~3.2 and~3.4 of~\cite{rbn}.
\end{proof}

The following lemmas give information about the spaces of marked BN-curves
passing through small numbers of points.

\begin{lm} \label{small-irred}
Let $\Gamma \subset \pp^r$ be a general set of $n \leq r + 2$ points,
and $d$ and $g$ be integers with $\rho(d, g, r) \geq 0$.
Then the space of marked BN-curves of degree $d$ and genus $g$ to $\pp^r$
passing through $\Gamma$ is irreducible.
\end{lm}
\begin{proof}
First note that, since $n \leq r + 2$, any $n$ points in linear general position
are related by an automorphism of $\pp^r$. Fix some ordering on $\Gamma$.

The space of BN-curves of degree $d$ and genus $g$
is irreducible, and the source of the generic BN-curve is irreducible;
consequently the space of such BN-curves with an ordered collection of $n$
marked points, and the open subset thereof where the images of the marked points
are in linear general position, is irreducible.
It follows that the space of such marked curves endowed with an automorphism
bringing the images of the ordered marked points to~$\Gamma$ (respecting our fixed ordering on $\Gamma$)
is also irreducible.
But by applying the automorphism to the curve and forgetting the order of the marked points,
this latter
space dominates the space of such BN-curves passing through~$\Gamma$;
the space of such BN-curves passing through~$\Gamma$ is thus irreducible.
\end{proof}

\begin{lm} \label{gen-tang-rat}
Let $\Gamma \subset \pp^r$ be a general set of $n \leq r + 2$ points, and
$\{\ell_p : p \in \Gamma\}$ be a set of lines with $p \in \ell_p$.

Then the general marked rational normal curve
passing through $\Gamma$ has tangent lines at each point $p \in \Gamma$ distinct from $\ell_p$.
\end{lm}
\begin{proof}
Since the intersection of dense opens is a dense open, it suffices to show
the general marked rational normal curve $(f \colon C \to \pp^r, \Delta)$ passing through $\Gamma$
has tangent line at $p$ distinct from $\ell_p$
for any one $p \in \Gamma$.

For this we consider the map, from the space of such marked rational normal curves, to the space
of lines through $p$, which associates to the curve its tangent line at $p$.
Basic deformation theory implies this map is smooth (and thus nonconstant) at $(f, \Delta)$
so long as $H^1(N_f(-\Delta)(-q)) = 0$, where $q \in \Delta$ is the point sent to $p$ under $f$,
which follows from combining Propositions~\ref{inter} and~\ref{twist}.
\end{proof}

\begin{lm} \label{contains-rat} A general BN-curve $f \colon C \to \pp^r$ can be specialized to an unramified map from a
reducible curve $f^\circ \colon X \cup_\Gamma Y \to \pp^r$,
where $f^\circ|_X$ is a rational normal curve.
\end{lm}
\begin{proof}
Write $d$ and $g$ for the degree and genus of $f$.
We first note it suffices to produce a marked WBN-curve $(f^\circ_2 \colon Y \to \pp^r, \Gamma_2)$ of degree $d - r$
and genus $g' \geq g - r - 1$, passing through a set
$\Gamma$ of $g + 1 - g'$ general points.
Indeed, $g + 1 - g' \leq g + 1 - (g - r - 1) = r + 2$ by assumption;
by Lemma~\ref{gen-tang-rat}, there is a marked rational normal curve $(f^\circ_1 \colon X \to \pp^r, \Gamma_1)$ passing through $\Gamma$,
whose tangent lines at $\Gamma$ are distinct from the tangent lines of $(f_2^\circ, \Gamma_2)$ at~$\Gamma$.
Then $f^\circ \colon X \cup_\Gamma Y \to \pp^r$ is unramified (as promised by our conventions)
and gives the required specialization by 
Proposition~\ref{prop:glue}.

It remains to construct $(f_2^\circ \colon Y \to \pp^r, \Gamma_2)$. If $g \leq r$, then we note that since
$d$ and $g$ are integers,
\[d \geq d - \frac{\rho(d, g, r)}{r + 1} = g + r - \frac{g}{r + 1} \imp d \geq g + r \quad \Leftrightarrow \quad g + 1 \leq (d - r) + 1.\]
Consequently, by inspection,
there is a marked rational curve $(f_2^\circ \colon Y \to \pp^r, \Gamma_2)$ of degree $d - r$ passing through a set $\Gamma$ of $g + 1$ general points.

On the other hand, if $g \geq r + 1$, then we
note that
\[\rho(d - r, g - r - 1, r) = (r + 1)(d - r) - r(g - r - 1) - r(r + 1) = (r + 1)d - rg - r(r + 1) = \rho(d, g, r) \geq 0.\]
We may therefore let $(f_2^\circ \colon Y \to \pp^r, \Gamma_2)$ be a marked BN-curve of degree $d - r$ and genus $g - r - 1$
passing through a set $\Gamma$ of $r + 2$ general points.
\end{proof}

\begin{lm} \label{gen-tang}
Let $\Gamma \subset \pp^r$ be a general set of $n \leq r + 2$ points,
$\{\ell_p : p \in \Gamma\}$ be a set of lines with $p \in \ell_p$,
and $d$ and $g$ be integers with $\rho(d, g, r) \geq 0$.

Then the general marked BN-curve $(f \colon C \to \pp^r, \Delta)$ of degree $d$ and genus $g$
passing through $\Gamma$
has tangent lines at every $p \in \Gamma$ which are distinct from $\ell_p$.
\end{lm}
\begin{proof}
By Lemma~\ref{contains-rat}, we may specialize $f \colon C \to \pp^r$
to $f^\circ \colon X \cup_\Gamma Y \to \pp^r$ where $f^\circ|_X$ is a rational
normal curve. Specializing the marked points $\Delta$ to lie on $X$
(which can be done since a marked rational normal curve can pass through $n \leq r + 2$ general points
by Proposition~\ref{inter}),
it suffices to consider the case when $f$ is a rational
normal curve.
But this case was already considered in Lemma~\ref{gen-tang-rat}.
\end{proof}

\begin{lm} \label{contains-rat-sp}
Lemma~\ref{contains-rat} remains true
even if we instead ask
$f^\circ|_X$ to be an arbitrary nondegenerate specialization
of a rational normal curve.
\end{lm}
\begin{proof}
We employ the construction used in the proof of Lemma~\ref{contains-rat},
but flipping the order in which we construct $X$ and $Y$:
First we fix $(f_1^\circ \colon X \to \pp^r, \Gamma_1)$; then we construct $(f_2^\circ \colon Y \to \pp^r, \Gamma_2)$
passing through $\Gamma$,
whose tangent lines at
$\Gamma$ are distinct from the tangent lines of $(f_1^\circ, \Gamma_1)$ at $\Gamma$
thanks to Lemma~\ref{gen-tang}.
\end{proof}

\section{Inductive Arguments \label{sec:indarg}}

Let $f \colon C \cup_u L \to \pp^r$ be an unramified map from a reducible curve,
with $L \simeq \pp^1$, and $u$ a single point,
and $f|_L$ of degree~$1$.
By Proposition~\ref{prop:glue}, these
curves are BN-curves.

\begin{lm} \label{p4-add-line} If $H^1(N_{f|_C}(-1)) = 0$,
then $H^1(N_f(-1)) = 0$.
\end{lm}
\begin{proof}
This is immediate from Lemma~\ref{glue} (taking $D = L$).
\end{proof}

\begin{lm} \label{p3-add-line} If $H^1(N_{f|_C}(-2)) = 0$,
and $f$ is a general map of the above type extending $f|_C$, then $H^1(N_f(-2)) = 0$.
\end{lm}
\begin{proof}
By Lemma~\ref{addone-raw}, it suffices to prove that
for $(u, v) \in C \times \pp^3$ general,
\[H^1(N_{f|_C}(-2)(u)[2u \to v]) = 0.\]
Since $H^1(N_{f|_C}(-2)) = 0$, we also have $H^1(N_{f|_C}(-2)(u)) = 0$;
in particular, Riemann-Roch implies
\begin{align*}
\dim H^0(N_{f|_C}(-2)(u)) &= \chi(N_{f|_C}(-2)(u)) = 2 \\
\dim H^0(N_{f|_C}(-2)) &= \chi(N_{f|_C}(-2)) = 0.
\end{align*}

The above dimension
estimates imply there is a unique section $s \in \pp H^0(N_{f|_C}(-2)(u))$
with $s|_u \in N_{f|_C \to v}|_u$; it remains to show that for $(u, v)$
general, $\langle s|_{2u} \rangle \neq N_{f|_C \to v}|_{2u}$.
For this, it suffices to verify that if $v_1$ and $v_2$
are points with $\{v_1, v_2, f(2u)\}$ coplanar --- but
neither $\{v_1, v_2, f(u)\}$, nor $\{v_1, f(2u)\}$, nor $\{v_2, f(2u)\}$
collinear; and $\{v_1, v_2, f(3u)\}$ not coplanar --- then
$N_{f|_C \to v_1}|_{2u} \neq N_{f|_C \to v_2}|_{2u}$.

To show this, we choose a local coordinate $t$ on $C$,
and coordinates on an appropriate affine open $\aa^3 \subset \pp^3$, so that:
\begin{align*}
f(t) &= (t, t^2 + O(t^3), O(t^3)) \\
v_1 &= (1 , 0 , 1) \\
v_2 &= (-1 , 0 , 1).
\end{align*}

It remains to check that the vectors
$f(t) - v_1$, $f(t) - v_2$, and $\frac{d}{dt} f(t)$
are linearly independent at first order in $t$. That is,
we want to check that the determinant
\[\left|\begin{array}{ccc}
t - 1 & t^2 + O(t^3) & O(t^3) - 1 \\
t + 1 & t^2 + O(t^3) & O(t^3) - 1 \\
1 & 2t + O(t^2) & O(t^2)
\end{array}\right|\not\equiv 0 \mod t^2.\]
Or, reducing the entries of the left-hand side modulo $t^2$, that
\[-4t = \left|\begin{array}{ccc}
t - 1 & 0 & - 1 \\
t + 1 & 0 & - 1 \\
1 & 2t & 0
\end{array}\right|\not\equiv 0 \mod t^2,\]
which is clear.
\end{proof}

\begin{lm} \label{add-can-3}
Let $\Gamma \subset \pp^3$ be a set of $5$ general points,
$(f_1 \colon C \to \pp^3, \Gamma_1)$ be a general marked BN-curve
passing through $\Gamma$, and
$(f_2 \colon D \to \pp^3, \Gamma_2)$
be a general marked canonical curve
passing through $\Gamma$.
If $H^1(N_{f_1}(-2)) = 0$,
then $f \colon C \cup_\Gamma D \to \pp^r$ satisfies $H^1(N_f(-2)) = 0$.
\end{lm}

\begin{rem}
By Lemma~\ref{small-irred}, it makes sense to speak of a
``general marked BN-curve (respectively general marked canonical curve)
passing through $\Gamma$'';
by Lemma~\ref{gen-tang}, the resulting curve $f$ is unramified.
\end{rem}

\begin{proof}
By Lemma~\ref{glue}, our problem reduces
to showing that the natural map
\[H^0(N_{f_2} (-2)) \to \bigoplus_{p \in \Gamma} \left(\frac{T_p (\pp^r)}{f_* (T_p (C \cup_\Gamma D))}\right)\]
is surjective, and that
\[H^1(N_f|_D (-2)) = 0.\]
These conditions both being open, we may invoke
Lemma~\ref{contains-rat} to specialize
$(f_1 \colon C \to \pp^3, \Gamma_1)$ to a marked BN-curve with reducible source
$(f_1^\circ \colon C_1 \cup_\Delta C_2 \to \pp^3, \Gamma_1^\circ)$,
with $f_1^\circ|_{C_1}$ a rational
normal curve and $\Gamma_1^\circ \subset C_1$.
It thus suffices to prove the above statements in the case when $f_1 = f_1^\circ$
is a rational normal curve.

For this, we first observe that $f(C) \cap f(D) = \Gamma$:
Since there is a unique rational normal curve through any $6$ points,
and a $1$-dimensional family of possible sixth points on $D$
once $D$ and $\Gamma$ are fixed --- but there is a $2$-dimensional family
of rational normal curves through $5$ points
in linear general position --- 
dimension counting shows $f_1(C)$ and $f_2(D)$ cannot meet at a sixth point
for $([f_1, \Gamma_1], [f_2, \Gamma_2])$ general.
In particular, $f$ is an immersion.

Next, we observe that $f(D)$ is contained in a $5$-dimensional space
of cubics. Since it is one linear condition, for a cubic that vanishes on $f(D)$,
to be tangent to $f(C)$ at a point of $\Gamma$, there is necessarily a cubic
surface $S$ containing $f(D)$ which is tangent to $f(C)$ at four points of $\Gamma$.

If $S$ were a multiple of $Q$, say $Q \cdot H$ where $H$ is a hyperplane, then
since $f(C)$ is transverse to $Q$, it would follow that $H$ contains four points of $\Gamma$.
But any $4$ points on $f(C)$ are in linear general position. Consequently, $S$ is not
a multiple of $Q$. Or equivalently, $f(D) = Q \cap S$ gives a presentation of $f(D)$
as a complete intersection.

If $S$ were tangent to $f(C)$ at all five points of $\Gamma$, then restricting the
equation of $S$ to $f(C)$ would give a section of $\oo_C(3) \simeq \oo_{\pp^1}(9)$
which vanished with multiplicity two at five points. Since the only such section
is the zero section, we would conclude that $f(C) \subset S$.
But then $f(C)$ would meet $f(D)$ at all $6$ points of $f(C) \cap Q$,
which we already ruled out above.
Thus, $S$ is tangent to $f(C)$ at precisely four points of $\Gamma$.

Write $\Delta$ for the divisor on $D$ defined by these four points,
and $p$ for the fifth point. Note that for $q \neq p$ in the tangent line to $(f_1, \Delta \cup \{p\})$
at $p$,
\begin{align*}
N_f|_D &\simeq \big(N_{f(D)/S}(\Delta + p) \oplus N_{f(D)/Q}(p)\big)[p \to q] \\
&\simeq \big(\oo_D(2)(\Delta + p) \oplus \oo_D(3)(p)\big)[p \to q] \\
\Rightarrow \ N_f|_D(-2) &\simeq \big(\oo_D(\Delta + p) \oplus \oo_D(1)(p)\big)[p \to q] \\
&\simeq \big(\oo_D(\Delta + p) \oplus K_D(p)\big)[p \to q].
\end{align*}

By Riemann-Roch, $\dim H^0(K_D(p)) = 4 = \dim H^0(K_D)$; so every section
of $K_D(p)$ vanishes at $p$. Consequently,
the fiber of every section of $\oo_D(\Delta + p) \oplus K_D(p)$
at $p$ lies in the fiber of the first factor. Since the fiber $N_{f_2 \to q}|_p$
does not lie in the fiber of the first factor, we have an isomorphism
\[H^0(N_f|_D(-2)) \simeq H^0\Big(\big(\oo_D(\Delta + p) \oplus K_D(p)\big)(-p)\Big) \simeq H^0(\oo_D(\Delta)) \oplus H^0(K_D).\]
Consequently,
\[\dim H^0(N_f|_D(-2)) = \dim H^0(\oo_D(\Delta)) + \dim H^0(K_D) = 1 + 4 = 5 = \chi(N_f|_D(-2)),\]
which implies
\[H^1(N_f|_D(-2)) = 0.\]

\noindent
Next, we prove the surjectivity of the evaluation map
\[\text{ev} \colon H^0(N_{f_2}(-2)) \to \bigoplus_{x \in \Gamma} \left(\frac{T_x (\pp^r)}{f_* (T_x (C \cup_\Gamma D))}\right)\]
For this, we use the isomorphism
\[N_{f_2}(-2) \simeq N_{f(D)/\pp^3}(-2) \simeq N_{f(D)/S}(-2) \oplus N_{f(D)/Q}(-2) \simeq \oo_D \oplus K_D.\]
The restriction of $\text{ev}$ to $H^0(N_{f(D)/S}(-2) \simeq \oo_D)$
maps trivially into the quotient $\frac{T_x (\pp^r)}{f_*(T_x (C \cup_\Gamma D))}$
for $x \in \Delta$, since $S$ is tangent to $f(C)$ along $\Delta$.
Because $S$ is not tangent to $f(C)$ at $p$,
the restriction of $\text{ev}$ to $H^0(N_{f(D)/S}(-2) \simeq \oo_D)$ thus
maps isomorphically onto the factor $\frac{T_p (\pp^r)}{f_*(T_p (C \cup_\Gamma D))}$.
It is therefore sufficient to show that the evaluation map
\[H^0(N_{f(D)/Q}(-2) \simeq K_D) \to \bigoplus_{x \in \Delta} \left(\frac{T_x (\pp^r)}{f_*(T_x (C \cup_\Gamma D))}\right)\]
is surjective. Or equivalently, since $Q$ is not tangent to $f(C)$ at any $x \in \Delta$,
that the evaluation map
\[H^0(K_D) \to K_D|_\Delta\]
is surjective. But this is clear since $\dim H^0(K_D) = 4 = \# \Delta$
and $\Delta$ is a general effective divisor of degree~$4$ on $D$.
\end{proof}

\begin{lm} \label{to-3-skew}
Let $f \colon C \to \pp^4$ be a general BN-curve in $\pp^4$, of arbitrary degree and genus.
Then we can specialize $f$ to an unramified map from a reducible curve
$f^\circ \colon C' \cup L_1 \cup L_2 \cup L_3 \to \pp^4$,
so that each $L_i$ is rational, $f^\circ|_{L_i}$ is of degree~$1$,
and the images of the $L_i$ under $f^\circ$ are in linear general position.
\end{lm}

\begin{proof}
By Lemma~\ref{contains-rat-sp},
our problem reduces to the case $f\colon C \to \pp^4$ is a rational normal curve.

In this case, we begin by taking three general lines in $\pp^4$.
The locus of lines meeting
each of our lines has class $\sigma_2$ in the Chow ring of
the Grassmannian $\mathbb{G}(1, 4)$ of lines in $\pp^4$.
By the standard calculus of Schubert cycles,
we have $\sigma_2^3 = \sigma_{2,2} \neq 0$
in the Chow ring of $\mathbb{G}(1, 4)$.
Thus, there exists a line meeting each of our three given lines.
The (immersion of the)
union of these four lines is then a specialization of a rational
normal curve.
\end{proof}

\begin{lm} \label{add-can-4}
Let $\Gamma \subset \pp^4$ be a set of $6$ points in linear general position;
$(f_1 \colon C \to \pp^4, \Gamma_1)$ be either a general marked
immersion of three disjoint lines,
or a general marked BN-curve in $\pp^4$, passing through $\Gamma$;
and $(f_2 \colon D \to \pp^4, \Gamma_2)$ be a general marked canonical curve
passing through~$\Gamma$.
If $H^1(N_{f_1}(-1)) = 0$, then
$f \colon C \cup_\Gamma D \to \pp^4$
satisfies $H^1(N_f(-1)) = 0$.
\end{lm}

\begin{proof}
By Lemma~\ref{glue}, it suffices to prove that the natural map
\[H^0(N_{f_2}(-1)) \to \bigoplus_{p \in \Gamma} \left(\frac{T_p(\pp^r)}{f_*(T_p(C \cup_\Gamma D))}\right)\]
is surjective, and that
\[H^1(N_f|_D(-1)) = 0.\]
These conditions both being open,
we may apply Lemma~\ref{to-3-skew}
to specialize $(f_1, \Gamma_1)$ to a marked curve
with reducible source
$(f_1^\circ \colon C_1 \cup C_2 \to \pp^r, \Gamma_1^\circ)$,
with $C_1 = L_1 \cup L_2 \cup L_3$ a union of $3$ disjoint lines,
and $\Gamma_1^\circ \subset C_1$ with $2$ points on each line.
It thus suffices to prove the above statements in the case when $C = C_1 = L_1 \cup L_2 \cup L_3$
is the union of $3$ general lines.
Write $\Gamma = \Gamma_1 \cup \Gamma_2 \cup \Gamma_3$, where $\Gamma_i \subset L_i$.

It is well known that every canonical curve in $\pp^4$ is the complete intersection of three quadrics;
write $V$ for the vector space of quadrics vanishing along $f(D)$.
For any $2$-secant line $L$ to $f(D)$, it is evident that it is one linear condition
on quadrics in $V$ to contain $L$; and moreover, that general lines impose independent
conditions unless there is a quadric which contains all $2$-secant lines.
Now the projection from a general line in $\pp^4$ of $f(D)$ yields a nodal plane curve
of degree $8$ and geometric genus $5$, which in particular must have
\[\binom{8 - 1}{2} - 5 = 16\]
nodes.
Consequently, the secant variety to $f(D)$
is a hypersurface of degree $16$; and is thus not contained in a quadric.
Thus, vanishing on general lines impose independent
conditions on~$V$. As $f(L_1)$, $f(L_2)$, and $f(L_3)$ are general,
we may thus choose a basis $V = \langle Q_1, Q_2, Q_3 \rangle$
so that $Q_i$ contains $L_j$ if an only if $i \neq j$
(where the $Q_i$ are uniquely defined up to scaling).
By construction, $f(D)$ is the complete intersection $Q_1 \cap Q_2 \cap Q_3$.

We now consider the direct sum decomposition
\[N_{f_2} \simeq N_{f(D)/\pp^4} \simeq N_{f(D)/(Q_1 \cap Q_2)} \oplus N_{f(D)/(Q_2 \cap Q_3)} \oplus N_{f(D)/(Q_3 \cap Q_1)},\]
which induces a direct sum decomposition
\[N_f|_D \simeq N_{f(D)/(Q_1 \cap Q_2)}(\Gamma_3) \oplus N_{f(D)/(Q_2 \cap Q_3)}(\Gamma_1) \oplus N_{f(D)/(Q_3 \cap Q_1)}(\Gamma_2).\]
To show that $H^1(N_f|_D(-1)) = 0$, it is sufficient
by symmetry to show that
\[H^1(N_{f(D)/(Q_1 \cap Q_2)}(\Gamma_3)(-1)) = 0.\]
But we have
\[N_{f(D)/(Q_1 \cap Q_2)}(\Gamma_3)(-1) \simeq \oo_D(2)(\Gamma_3)(-1) \simeq \oo_D(1)(\Gamma_3) = K_D(\Gamma_3);\]
so by Serre duality,
\[H^1(N_{f(D)/(Q_1 \cap Q_2)}(\Gamma_3)(-1)) \simeq H^0(\oo_D(-\Gamma_3))^\vee = 0.\]

\noindent
Next, we examine the evaluation map
\[H^0(N_{f_2}(-1)) \to \bigoplus_{p \in \Gamma} \left(\frac{T_p(\pp^r)}{f_*(T_p(C \cup_\Gamma D))}\right).\]
For this, we use the direct sum decomposition
\[N_{f_2} \simeq N_{f(D)/\pp^4} \simeq N_{f(D)/(Q_1 \cap Q_2)}(-1) \oplus N_{f(D)/(Q_2 \cap Q_3)}(-1) \oplus N_{f(D)/(Q_3 \cap Q_1)}(-1),\]
together with the decomposition (for $p \in \Gamma_i$):
\[\frac{T_p (\pp^r)}{f_*(T_p(C \cup_{\Gamma_i} L_i))} \simeq \bigoplus_{j \neq i} N_{f(D)/(Q_i \cap Q_j)}|_p.\]
This reduces our problem to showing (by symmetry) the surjectivity of
\[H^0(N_{f(D)/(Q_1 \cap Q_2)}(-1)) \to \bigoplus_{p \in \Gamma_1 \cup \Gamma_2} N_{f(D)/(Q_1 \cap Q_2)}|_p.\]
But for this, it is sufficient to note that
$\Gamma_1 \cup \Gamma_2$ is a general collection of $4$ points
on $D$, and
\[N_{f(D)/(Q_1 \cap Q_2)}(-1) \simeq \oo_D(2)(-1) = \oo_D(1) \simeq K_D.\]
It thus remains to show
\[H^0(K_D) \to K_D|_{\Gamma_1 \cup \Gamma_2}\]
is surjective, where $\Gamma_1 \cup \Gamma_2$ is a general collection of $4$ points
on $D$. But this is clear because $K_D$ is a line bundle
and $\dim H^0(K_D) = 5 \geq 4$.
\end{proof}

\begin{cor} \label{finite} To prove the main theorems (excluding the ``conversely\ldots'' part),
it suffices to verify them in the following special cases:
\begin{enumerate}
\item For Theorem~\ref{main-3}, it suffices to consider the cases where $(d, g)$ is one of:
\begin{gather*}
(5, 1), \quad (7, 2), \quad (6, 3), \quad (7, 4), \quad (8, 5), \quad (9, 6), \quad (9, 7), \\
(10, 9), \quad (11, 10), \quad (12, 12), \quad (13, 13) \quad (14, 14).
\end{gather*}
\item For Theorem~\ref{main-3-1}, it suffices to consider the cases where $(d, g)$ is one of:
\[(7, 5), \quad (8, 6).\]
\item For Theorem~\ref{main-4}, it suffices to consider the cases where $(d, g)$ is one of:
\[(9, 5), \quad (10, 6), \quad (11, 7), \quad (12, 9), \quad (16, 15), \quad (17, 16), \quad (18, 17).\]
\end{enumerate}
In proving the theorems in each of these cases, we may suppose the
corresponding theorem holds for curves of smaller genus.
\end{cor}
\begin{proof}
For Theorem~\ref{main-3}, note that by Lemma~\ref{p3-add-line}
and Proposition~\ref{prop:glue}, it suffices to show Theorem~\ref{main-3} for
each pair $(d, g)$, where $d$ is minimal (i.e.,\ where $\rho(d, g) = \rho(d, g, r = 3) \geq 0$
and $(d, g)$ is not in our list of counterexamples; but either $\rho(d - 1, g) < 0$,
or $(d - 1, g)$ is in our list of counterexamples).

If $\rho(d, g) \geq 0$ and $g \geq 15$, then $(d - 6, g - 8)$ is not in our list of counterexamples,
and $\rho(d - 6, g - 8) = \rho(d, g) \geq 0$. By induction, we know $H^1(N_f(-2)) = 0$
for $f$ a BN-general curve of degree $d - 6$ and genus $g - 8$.
Applying Lemma~\ref{add-can-3} (and Proposition~\ref{prop:glue}), we conclude the desired result.
If $\rho(d, g) \geq 0$ and $g \leq 14$, and $d$ is minimal as above,
then either $(d, g)$ is in our above list, or
$(d, g) \in \{(3, 0), (9, 8), (12, 11)\}$. The case of $(d, g) = (3, 0)$ follows from
Lemma~\ref{from-inter}.
But in these last two cases,
Lemma~\ref{add-can-3} again implies the desired result (using Theorem~\ref{main-3}
for $(d', g') = (d - 6, g - 8)$ as our inductive hypotheses).

For Theorem~\ref{main-3-1}, we note that if $H^1(N_f(-2)) = 0$, then
it follows that $H^1(N_f(-1)) = 0$. It therefore suffices to check
the list of counterexamples appearing in Theorem~\ref{main-3}
besides the counterexample $(d, g) = (6, 4)$ listed in Theorem~\ref{main-3-1}.
The cases $(d, g) \in \{(4, 1), (5, 2), (6, 2)\}$ follow from Lemma~\ref{from-inter},
so we only have to consider the remaining cases (which form the given list).

Finally, for Theorem~\ref{main-4}, Lemma~\ref{p4-add-line} implies it suffices
to show Theorem~\ref{main-4} for each pair $(d, g)$ with $d$ minimal.
If $\rho(d, g) \geq 0$ and $g \geq 18$, then $(d - 8, g - 10)$ is not in our list of counterexamples,
and $\rho(d - 8, g - 10) = \rho(d, g) \geq 0$. By induction, we know $H^1(N_f(-1)) = 0$
for $C$ is a general curve of degree $d - 8$ and genus $g - 10$.
Applying Lemma~\ref{add-can-4}, we conclude the desired result.
If $\rho(d, g) \geq 0$ and $g \leq 17$, and $d$ is minimal as above,
then either $(d, g)$ is in our above list, or
\[(d, g) \in \{(4, 0), (5, 1), (6, 2), (7, 3), (8, 4)\},\]
or
\[(d, g) \in \{(11, 8), (12, 10), (13, 11), (14, 12), (15, 13), (16, 14)\},\]

In the first set of cases above, Lemma~\ref{from-inter} implies the desired
result. But in the last set of cases,
Lemma~\ref{add-can-4} again implies the desired result. Here, for $(d, g) = (11, 8)$,
our inductive hypothesis is that
$H^1(N_f(-1)) = 0$ for $f \colon L_1 \cup L_2 \cup L_3 \to \pp^4$
an immersion of three skew lines.
In the remaining cases, we use Theorem~\ref{main-3}
for $(d', g') = (d - 8, g - 10)$ as our inductive hypothesis.
\end{proof}

\section{Adding Curves in a Hyperplane \label{sec:hir}}

In this section, we explain an inductive strategy involving adding
curves contained in hyperplanes, which will help resolve many of our
remaining cases.

\begin{lm} \label{smoothable} Let $H \subset \pp^r$ (for $r \geq 3$) be a hyperplane,
and let $(f_1 \colon C \to \pp^r, \Gamma_1)$ and
\mbox{$(f_2 \colon D \to H, \Gamma_2)$} be marked curves,
both passing through a set $\Gamma \subset H \subset \pp^r$ of $n \geq 1$ points.

Assume that $f_2$ is a general BN-curve of degree $d$ and genus $g$ to $H$,
that $\Gamma_2$ is a general collection of $n$ points on $D$, and that $f_1$ is transverse
to $H$ along $\Gamma$. If
\[H^1(N_{f_1}(-\Gamma)) = 0 \quad \text{and} \quad n \geq g - d + r,\]
then $f \colon C \cup_\Gamma D \to \pp^r$
satisfies $H^1(N_f) = 0$
and is a limit of unramified maps from smooth curves.

If in addition $f_1$ is an immersion,
$f(C) \cap f(D)$ is exactly equal to $\Gamma$, and
$\oo_D(1)(\Gamma)$ is very ample away from $\Gamma$ --- i.e.\ if
$\dim H^0(\oo_D(1)(\Gamma)(-\Delta)) = \dim H^0(\oo_D(1)(\Gamma)) - 2$
for any effective divisor $\Delta$ of degree $2$ supported on $D \smallsetminus \Gamma$ --- then
$f$ is a limit of immersions of smooth curves.
\end{lm}

\begin{rem} \label{very-ample-away}
The condition that $\oo_D(1)(\Gamma)$ is very ample away from $\Gamma$
is immediate when $\oo_D(1)$ is very ample (which in particular happens for $r \geq 4$).
It is also immediate when $n \geq g$, in which case $\oo_D(1)(\Gamma)$ is a general line bundle
of degree $d + n \geq g + r \geq g + 3$ and is thus very ample.
\end{rem}

\begin{proof}
Note that $N_{f_1}$ is a subsheaf of $N_f|_C$ with punctual quotient
(supported at $\Gamma$). Twisting down by $\Gamma$, we obtain a short exact sequence
\[0 \to N_{f_1}(-\Gamma) \to N_f|_C(-\Gamma) \to * \to 0,\]
where $*$ denotes a punctual sheaf, which in particular has vanishing $H^1$.
Since $H^1(N_{f_1}(-\Gamma)) = 0$ by assumption,
we conclude that $H^1(N_f|_C(-\Gamma)) = 0$ too.
Since $f_2$ is a general BN-curve, $H^1(N_{f_2}) = 0$.
The exact sequences
\begin{gather*}
0 \to N_f|_C(-\Gamma) \to N_f \to N_f|_D \to 0 \\
0 \to N_{f_2} \to N_f|_D \to N_H|_D(\Gamma) \simeq \oo_D(1)(\Gamma) \to 0
\end{gather*}
then imply that, to check $H^1(N_f) = 0$, it suffices to check $H^1(\oo_D(1)(\Gamma)) = 0$.
They moreover imply that
every section of $N_H|_D(\Gamma) \simeq \oo_D(1)(\Gamma)$ lifts to a section
of $N_f$, which, as $H^1(N_f) = 0$, lifts to a global deformation
of $f$.

To check $f$
is a limit of unramified maps from smooth curves, it remains to see that the
generic section of $N_H|_D(\Gamma) \simeq \oo_D(1)(\Gamma)$ corresponds
to a first-order deformation which smoothes the nodes $\Gamma$ --- or equivalently does not vanish at $\Gamma$.

Since by assumption $f_1$ is an immersion
and there are no other nodes where $f(C)$ and $f(D)$ meet besides $\Gamma$,
to see that $f$
is a limit of immersions of smooth curves, it remains to note in addition that
the generic section of $N_H|_D(\Gamma) \simeq \oo_D(1)(\Gamma)$
separates the points of $D$ identified under $f_2$ --- which is true by assumption that $\oo_D(1)(\Gamma)$ is very ample
away from $\Gamma$.

To finish the proof, it thus suffices to check $H^1(\oo_D(1)(\Gamma)) = 0$,
and that the generic section of $\oo_D(1)(\Gamma)$ does not vanish at any point $p \in \Gamma$.
Equivalently, it suffices to check $H^1(\oo_D(1)(\Gamma)(-p)) = 0$ for $p \in \Gamma$.
Since $f_2$ is a general BN-curve, we obtain
\[\dim H^1(\oo_D(1)) = \max(0, g - d + (r - 1)) \leq n - 1.\]
Twisting by $\Gamma \smallsetminus \{p\}$, which is a set of $n - 1$ general points, we therefore obtain
\[H^1(\oo_D(1)(\Gamma \smallsetminus \{p\})) = 0,\]
as desired.
\end{proof}

\begin{lm} \label{lm:hir}
Let $k \geq 1$ be an integer, $\iota \colon H \hookrightarrow \pp^r$ ($r \geq 3$) be a hyperplane,
and $(f_1 \colon C \to \pp^r, \Gamma_1)$ and
\mbox{$(f_2 \colon D \to H, \Gamma_2)$} be marked curves,
both passing through a set $\Gamma \subset H \subset \pp^r$ of $n \geq 1$ points.

Assume that $f_2$ is a general BN-curve of degree $d$ and genus $g$ to $H$,
that $\Gamma_2$ is a general collection of $n$ points on $D$, and that $f_1$ is transverse
to $H$ along $\Gamma$.
Suppose moreover that:
\begin{enumerate}
\item The bundle $N_{f_2}(-k)$ satisfies interpolation.
\item We have 
$H^1(N_{f_1}(-k)) = 0$.
\item We have
\[(r - 2) n \leq rd - (r - 4)(g - 1) - k \cdot (r - 2) d.\]
\item We have
\[n \geq \begin{cases}
g & \text{if $k = 1$;} \\
g - 1 + (k - 1)d & \text{if $k > 1$.}
\end{cases}\]
\end{enumerate}
Then $f \colon C \cup_\Gamma D \to \pp^r$ satisfies
\[H^1(N_f(-k)) = 0.\]
\end{lm}

\begin{proof}
Since $N_{f_2}(-k)$ satisfies interpolation by assumption and
\[(r - 2) n \leq \chi(N_{f_2}(-k)) = rd - (r - 4)(g - 1) - k \cdot (r - 2) d,\]
we conclude that $H^1(N_{f_2}(-k)(-\Gamma)) = 0$.
Since $H^1(N_{f_1} (-k)) = 0$ by assumption,
to apply Lemma~\ref{hyp-glue} it remains to check
\[H^1(\oo_D(1 - k)(\Gamma)) = 0.\]
It is therefore sufficient for
\[n = \#\Gamma \geq \dim H^1(\oo_D(1 - k)) = \begin{cases}
g & \text{if $k = 1$;} \\
g - 1 + (k - 1)d & \text{if $k > 1$.}
\end{cases}\]
But this is precisely our final assumption.
\end{proof}

\section{Curves of Large Genus \label{sec:hir-2}}

In this section, we will deal with a number of our special
cases, of larger genus. Taking care of these cases separately
is helpful --- since in the remaining cases, we will not
have to worry about whether our curve is a BN-curve, thanks to
results of~\cite{iliev} and~\cite{keem} on the irreducibility
of the Hilbert scheme of curves.

\begin{lm} \label{bn3}
Let $H \subset \pp^3$ be a plane, $\Gamma \subset H \subset \pp^3$ a set of $6$ general points,
$(f_1 \colon C \to \pp^3, \Gamma_1)$ a general marked BN-curve passing through $\Gamma$
of degree and genus one of
\[(d, g) \in \{(6, 1), (7, 2), (8, 4), (9, 5), (10, 6)\},\]
and $(f_2 \colon D \to H, \Gamma_2)$
a general marked canonical curve
passing through $\Gamma$.
Then $f \colon C \cup_\Gamma D \to \pp^3$ is a BN-curve which satisfies $H^1(N_f) = 0$.
\end{lm}
\begin{proof}
Note that the conclusion is an open condition; we may therefore freely specialize $(f_1, \Gamma_1)$.
Write $\Gamma = \{s, t, u, v, w, x\}$.

In the case $(d, g) = (6, 1)$, we specialize $(f_1, \Gamma_1)$
to 
$(f_1^\circ \colon C^\circ = C_1 \cup_p C_2 \cup_{\{q, r\}} C_3 \to \pp^3, \Gamma_1^\circ)$,
where $f_1^\circ|_{C_1}$ is a conic, $f_1^\circ|_{C_2}$ is a line with $C_2$ joined to $C_1$
at one point $p$, and $f_1^\circ|_{C_3}$ is a rational normal curve with $C_3$ joined to $C_1$ at two points $\{q, r\}$;
note that $f_1^\circ$ is a BN-curve by (iterative application of) Proposition~\ref{prop:glue}.
We suppose that $(f_1^\circ|_{C_1}, \Gamma_1^\circ \cap C_1)$ passes through $\{s, t\}$,
while $(f_1^\circ|_{C_2}, \Gamma_1^\circ \cap C_2)$ passes through $u$,
and $(f_1^\circ|_{C_3}, \Gamma_1^\circ \cap C_3)$ passes through $\{v, w, x\}$;
it is clear this can be done so $\{s, t, u, v, w, x\}$ are general.
Writing
\[f^\circ \colon C^\circ \cup_\Gamma D = C_2 \cup_{\{p, u\}} C_3 \cup_{\{q, r, v, w, x\}} (C_1 \cup_{\{s, t\}} D) \to \pp^3,\]
it suffices by Propositions~\ref{prop:glue} and~\ref{prop:interior} to
show that $f^\circ|_{C_1 \cup D}$ is a BN-curve which satisfies $H^1(N_{f^\circ|_{C_1 \cup D}}) = 0$.

For $(d, g) = (8, 4)$, we specialize $(f_1, \Gamma_1)$ to
$(f_1^\circ \colon C^\circ = C_1 \cup_{\{p, q, r\}} C_2 \cup_{\{y, z, a\}} C_3 \to \pp^3, \Gamma_1^\circ)$,
where $f_1^\circ|_{C_1}$ is a conic, and
$f_1^\circ|_{C_2}$ and $f_1^\circ|_{C_3}$ are rational normal curves,
with both $C_2$ and $C_3$ joined to $C_1$ at $3$ points (at $\{p, q, r\}$ and $\{y, z, a\}$ respectively);
note that $f_1^\circ$ is a BN-curve by (iterative application of) Proposition~\ref{prop:glue}.
We suppose that $(f_1^\circ|_{C_1}, \Gamma_1^\circ \cap C_1)$ passes through $\{s, t\}$,
while $(f_1^\circ|_{C_2}, \Gamma_1^\circ \cap C_2)$ passes through $\{u, v\}$,
and $(f_1^\circ|_{C_3}, \Gamma_1^\circ \cap C_3)$ passes through $\{w, x\}$;
it is clear this can be done so $\{s, t, u, v, w, x\}$ are general.
Writing
\[f^\circ \colon C^\circ \cup_\Gamma D = C_2 \cup_{\{p, q, r, u, v\}} C_3 \cup_{\{w, x, y, z, a\}} (C_1 \cup_{\{s, t\}} D) \to \pp^3,\]
it again suffices by Propositions~\ref{prop:glue} and~\ref{prop:interior} to
show that $f^\circ|_{C_1 \cup D}$ is a BN-curve which satisfies $H^1(N_{f^\circ|_{C_1 \cup D}}) = 0$.

For this, we first note that $f^\circ|_{C_1 \cup D}$ is a curve of degree $6$ and genus $4$,
and that the moduli space of smooth curves of degree $6$ and genus $4$ in $\pp^3$ is
irreducible (they are all canonical curves).
Moreover, by Lemma~\ref{smoothable} (c.f.\ Remark~\ref{very-ample-away} and note that
$\oo_D(1) \simeq K_D$ is very ample),
$f^\circ|_{C_1 \cup D}$ is a limit
of immersions of smooth curves, and satisfies
$H^1(N_{f^\circ|_{C_1 \cup D}}) = 0$; this completes the proof.
\end{proof}

\begin{lm} \label{bn4}
Let $H \subset \pp^4$ be a hyperplane, $\Gamma \subset H \subset \pp^4$ a set of $7$ general points,
\mbox{$(f_1 \colon C \to \pp^4, \Gamma_1)$} a general marked BN-curve passing through $\Gamma$
of degree and genus one of
\[(d, g) \in \{(7, 3), (8, 4), (9, 5)\},\]
and $(f_2 \colon D \to H, \Gamma_2)$
a general marked BN-curve of degree~$9$ and genus~$6$
passing through $\Gamma$.
Then $f \colon C \cup_\Gamma D \to \pp^4$ is a BN-curve which satisfies $H^1(N_f) = 0$.
\end{lm}
\begin{proof}
Again, we note that the conclusion is an open statement; we may therefore freely
specialize $(f_1, \Gamma_1)$. Write $\Gamma = \{t, u, v, w, x, y, z\}$.

First, we claim it suffices to consider the case $(d, g) = (7, 3)$.
Indeed, suppose $(f_1, \Gamma_1)$ is a marked BN-curve of degree $7$ and genus $3$ passing through $\Gamma$.
Then $f_1' \colon C \cup_{\{p, q\}} L \to \pp^4$ and $f_1'' \colon C \cup_{\{p, q\}} L \cup_{\{r, s\}} L' \to \pp^4$
(where $f_1'|_L$ and $f_1''|_L$ and $f_1''|_{L'}$ are lines with $L$ and $L'$ joined to $C$ at two points)
are BN-curves by Proposition~\ref{prop:glue}, of degree and genus
$(8, 4)$ and $(9, 5)$ respectively. If $f \colon C \cup_\Gamma D \to \pp^4$ is a BN-curve
with $H^1(N_f) = 0$, then invoking Propositions~\ref{prop:glue}
and~\ref{prop:interior}, both
\begin{gather*}
f' \colon (C \cup_{\{p, q\}} L) \cup_\Gamma D = (C \cup_\Gamma D) \cup_{\{p, q\}} L \to \pp^4 \\
\text{and} \quad f'' \colon (C \cup_{\{p, q\}} L \cup_{\{r, s\}} L') \cup_\Gamma D = (C \cup_\Gamma D) \cup_{\{p, q\}} L \cup_{\{r, s\}} L' \to \pp^4
\end{gather*}
are BN-curves, which satisfy
$H^1(N_{f'}) = H^1(N_{f''}) = 0$.

So it remains to consider the case $(d, g) = (7, 3)$.
In this case, we begin by specializing $(f_1, \Gamma_1)$ to
$(f_1^\circ \colon C^\circ = C' \cup_{\{p, q\}} L \to \pp^4, \Gamma_1^\circ)$,
where $f_1^\circ|_{C'}$ is a general BN-curve of degree $6$ and genus $2$, and $f_1^\circ|_L$ is a line with $L$
joined to $C'$ at two points $\{p, q\}$.
We suppose that $(f_1^\circ|_L, \Gamma_1^\circ \cap L)$ passes through $t$,
while $(f_1^\circ|_{C'}, \Gamma_1^\circ \cap C')$ passes through $\{u, v, w, x, y, z\}$;
we must check this can be done so $\{t, u, v, w, x, y, z\}$ are general.
To see this, it suffices to show
that the intersection $f_1^\circ(C') \cap H$ and the points $\{f_1^\circ(p), f_1^\circ(q)\}$
independently general. In other words,
we are claiming that the map
\[\{(f_1^\circ|_{C'} \colon C' \to \pp^4, p, q) : p, q \in C'\} \mapsto (f_1^\circ|_{C'}(C') \cap H, f_1^\circ|_{C'}(p), f_1^\circ|_{C'}(q))\]
is dominant; equivalently, that it is smooth at a generic point $(f_1^\circ|_{C'}, p, q)$.
But the obstruction to smoothness lies in
$H^1(N_{f_1^\circ|_{C'}}(-1)(-p-q)) = 0$, which vanishes because
because $N_{f_1^\circ|_{C'}}(-1)$ satisfies interpolation by Lemma~\ref{from-inter}.

We next specialize $(f_2, \Gamma_2)$ to $(f_2^\circ \colon D^\circ = D' \cup_\Delta D_1 \to H, \Gamma_2^\circ)$,
where $f_2^\circ|_{D'}$ is a general BN-curve of degree
$6$ and genus $3$, and $f_2^\circ|_{D_1}$ is a rational normal curve with $D_1$
joined to $D'$ at a set $\Delta$ of $4$ points;
note that $f_2^\circ$ is a BN-curve by Proposition~\ref{prop:glue}.
We suppose that $(f_2^\circ|_{D_1}, \Gamma_2^\circ \cap D_1)$ passes through $t$,
while $(f_2^\circ|_{D'}, \Gamma_2^\circ \cap D')$ passes through $\{u, v, w, x, y, z\}$;
this can be done so $\{t, u, v, w, x, y, z\}$ are still general,
since $f_2^\circ|_{D'}$ (marked at general points of the source) can pass through $6$ general points,
while $(f_2^\circ|_{D_1}$ (again marked at general points of the source) can pass through $5$ general points,
both by Corollary~1.4 of~\cite{firstpaper}.

In addition, $(f_2^\circ|_{D_1}, (\hat{t} = \Gamma_2^\circ \cap D_1) \cup \Delta)$ has a general tangent line at $t$;
to see this, note that we are asserting that the map sending $(f_2^\circ|_{D_1}, \hat{t} \cup \Delta)$
to its tangent line at $t$ is dominant;
equivalently, that it is smooth at a generic point of the source.
But the obstruction to smoothness lies in
$H^1(N_{f_2^\circ|_{D_1}}(-\Delta - 2\hat{t} \, ))$, which vanishes because
$N_{f_2^\circ|_{D_1}}(-2\hat{t} \, )$ satisfies interpolation by combining
Propositions~\ref{inter} and~\ref{twist}.

As $\{p, q\} \subset C'$ is general, we thus know that the tangent lines
to $(f_2^\circ|_{D_1}, \hat{t} \cup \Delta)$ at $t$, and to $(f_1^\circ|_{C'}, \{p, q\})$ at $f_1^\circ(p)$ and $f_1^\circ(q)$,
together span all of $\pp^4$; write $\bar{t}$, $\bar{p}$, and $\bar{q}$ for points on each of these tangent lines
distinct from $t$, $f_1^\circ(p)$, and $f_1^\circ(q)$ respectively.
We then use the exact sequences
\begin{gather*}
0 \to N_{f^\circ}|_L(-\hat{t} - p - q) \to N_{f^\circ} \to N_{f^\circ}|_{C' \cup D^\circ} \to 0 \\
0 \to N_{f^\circ|_{C' \cup D^\circ}} \to N_{f^\circ}|_{C' \cup D^\circ} \to * \to 0,
\end{gather*}
where $*$ is a punctual sheaf (which in particular has vanishing $H^1$).
Write $H_t$ for the hyperplane spanned by $f_1^\circ(L)$, $\bar{p}$, and $\bar{q}$;
and $H_p$ for the hyperplane spanned by $f_1^\circ(L)$, $\bar{t}$, and $\bar{q}$;
and $H_q$ for the hyperplane spanned by $f_1^\circ(L)$, $\bar{t}$, and $\bar{p}$.
Then $f_1^\circ(L)$ is the complete intersection $H_t \cap H_p \cap H_q$, and so we get a decomposition
\[N_{f^\circ}|_L \simeq N_{f_1^\circ(L) / H_t}(\hat{t} \, ) \oplus N_{f_1^\circ(L) / H_p}(p) \oplus N_{f_1^\circ(L) / H_q}(q),\]
which upon twisting becomes
\[N_{f^\circ}|_L(-\hat{t} - p - q) \simeq N_{f_1^\circ(L) / H_t}(-p-q) \oplus N_{f_1^\circ(L) / H_p}(-\hat{t}-q) \oplus N_{f_1^\circ(L) / H_q}(-\hat{t} - p).\]
Note that $N_{f_1^\circ(L) / H_t}(-p-q) \simeq \oo_L(-1)$ has vanishing $H^1$, and similarly for the other factors;
consequently, $H^1(N_{f^\circ}|_L(-\hat{t} - p - q)) = 0$. We conclude that
$H^1(N_{f^\circ}) = 0$ provided that $H^1(N_{f^\circ|_{C' \cup D^\circ}}) = 0$.
Moreover, 
writing $C' \cup_{\{u, v, w, x, y, z\}} D^\circ = D_1 \cup_\Delta (D' \cup_{\{u, v, w, x, y, z\}} C')$
and applying Proposition~\ref{prop:interior}, we know that $H^1(N_{f^\circ|_{C' \cup D^\circ}}) = 0$ provided that
$H^1(N_{f^\circ|_{C' \cup D'}}) = 0$.
And if $f^\circ|_{C' \cup D'}$ is a BN-curve,
then
$f^\circ \colon (C' \cup_{\{u, v, w, x, y, z\}} D') \cup_{\Delta \cup \{p, q\}} (D_1 \cup_t L) \to \pp^4$
is a BN-curve too by Proposition~\ref{prop:glue},
Putting this all together, it is sufficient to show that
$f^\circ|_{C' \cup D'}$ is a BN-curve which satisfies $H^1(N_{f^\circ|_{C' \cup D'}}) = 0$.

Our next step is to specialize $(f_1^\circ|_{C'}, \Gamma_1^\circ \cap C')$ to
$(f_1^{\circ\circ} \colon C^{\circ\circ} = C'' \cup_{\{r, s\}} L' \to \pp^4, \Gamma_1^{\circ\circ})$,
where $f_1^{\circ\circ}|_{C''}$
is a general BN-curve of degree~$5$ and genus~$1$, and $f_1^{\circ\circ}|_{L'}$ is a line
with $L'$ joined to $C''$ at two points $\{r, s\}$.
We suppose that $(f_1^{\circ\circ}|_{C''}, \Gamma_1^{\circ\circ} \cap C'')$ passes through $u$,
while $(f_1^{\circ\circ}|_{C''}, \Gamma_1^{\circ\circ} \cap C'')$ passes through $\{v, w, x, y, z\}$;
as before this can be done so $\{u, v, w, x, y, z\}$ are general.
We also specialize $(f_2^\circ|_{D'}, \Gamma_2^\circ \cap D')$ to
$(f_2^{\circ\circ} \colon D'' \cup_\Delta D_2 \to \pp^4, \Gamma_2^{\circ\circ})$,
where $f_2^{\circ\circ}|_{D''}$ and $f_2^{\circ\circ}|_{D_2}$ are both rational normal curves
with $D''$ and $D_2$ joined at a set $\Delta$ of $4$ general points.
We suppose that $(f_2^{\circ\circ}|_{D_2}, \Gamma_2^{\circ\circ} \cap D_2)$
passes through $u$,
while $(f_2^{\circ\circ}|_{D''}, \Gamma_2^{\circ\circ} \cap D'')$
passes through $\{v, w, x, y, z\}$;
as before this can be done so $\{u, v, w, x, y, z\}$ are general.
The same argument as above, mutatis mutandis, then implies it is sufficient to show that
$f^{\circ\circ}|_{C'' \cup D''} \colon C'' \cup_{\{v, w, x, y, z\}} D'' \to \pp^4$ is a BN-curve which satisfies
$H^1(N_{f^{\circ\circ}|_{C'' \cup D''}}) = 0$.

For this, we first note that $f^{\circ\circ}|_{C'' \cup D''}$ is a curve of degree $8$ and genus $5$,
and that the moduli space of smooth curves of degree $8$ and genus $5$ in $\pp^4$ is
irreducible (they are all canonical curves).
To finish the proof, it suffices to note by Lemma~\ref{smoothable} that
$f^{\circ\circ}|_{C'' \cup D''}$ is a limit of immersions of smooth curves and satisfies
$H^1(N_{f^{\circ\circ}|_{C'' \cup D''}}) = 0$.
\end{proof}

\begin{cor} \label{smooth-enough} To prove the main theorems (excluding the ``conversely\ldots'' part),
it suffices to show the existence of (nondegenerate immersions of) smooth curves, of the following degrees
and genera, which satisfy the conclusions:
\begin{enumerate}
\item For Theorem~\ref{main-3}, it suffices to show the existence of smooth curves, satisfying
the conclusions, where $(d, g)$ is one of:
\[(5, 1), \quad (7, 2), \quad (6, 3), \quad (7, 4), \quad (8, 5), \quad (9, 6), \quad (9, 7).\]
\item For Theorem~\ref{main-3-1}, it suffices to show the existence of smooth curves, satisfying
the conclusions, where $(d, g)$ is one of:
\[(7, 5), \quad (8, 6).\]
\item For Theorem~\ref{main-4}, it suffices to show the existence of smooth curves, satisfying
the conclusions, where $(d, g)$ is one of:
\[(9, 5), \quad (10, 6), \quad (11, 7), \quad (12, 9).\]
\end{enumerate}
(And in constructing the above smooth curves, we may suppose the
corresponding theorem holds for curves of smaller genus.)
\end{cor}
\begin{proof}
By Lemmas~\ref{bn3} and~\ref{lm:hir}, and Proposition~\ref{p2}, we know that Theorem~\ref{main-3}
holds for $(d, g)$ one of
\[(10, 9), \quad (11, 10), \quad (12, 12), \quad (13, 13), \quad (14, 14).\]
Similarly, by Lemmas~\ref{bn4}, \ref{lm:hir}, and~\ref{from-inter}, we know that Theorem~\ref{main-4}
holds for $(d, g)$ one of
\[(16, 15), \quad (17, 16), \quad (18, 17).\]
Eliminating these cases from the lists in Corollary~\ref{finite},
we obtain the given lists of pairs $(d, g)$.

Moreover --- in each of the cases appearing in the statement
of this corollary --- results of \cite{keem} (for $r = 3$) and \cite{iliev} (for $r = 4$)
state that the Hilbert scheme of curves of degree $d$ and genus $g$ in $\pp^r$
has a \emph{unique} component whose points represent smooth irreducible nondegenerate curves.
The condition that our curve be a BN-curve may thus be replaced
with the condition that our curve be smooth irreducible nondegenerate.
\end{proof}

\section{More Curves in a Hyperplane \label{sec:hir-3}}

In this section, we give several more applications
of the technique developed in the previous two sections. Note that from
Corollary~\ref{smooth-enough},
it suffices to show the existence of curves satisfying
the desired conclusions which are limits of immersions of smooth curves;
it not necessary to check that these
curves are BN-curves.

\begin{lm} \label{lm:ind:3} Suppose $N_f(-2)$ satisfies interpolation, where $f \colon C \to \pp^3$ is a general BN-curve
of degree $d$ and genus $g$ to $\pp^3$. Then the same is true for some smooth curve of
degree and genus:
\begin{enumerate}
\item \label{33} $(d + 3, g + 3)$ (provided $d \geq 3$);
\item \label{42} $(d + 4, g + 2)$ (provided $d \geq 3$);
\item \label{46} $(d + 4, g + 6)$ (provided $d \geq 5$).
\end{enumerate}
\end{lm}
\begin{proof}
We apply Lemma~\ref{lm:hir} for $f_2$ a curve of degree up to $4$ (and note that
$N_{f_2}(-2)$ satisfies interpolation by Proposition~\ref{p2}), namely:
\begin{enumerate}
\item $(d_2, g_2) = (3, 1)$ and $n = 3$;
\item $(d_2, g_2) = (4, 0)$ and $n = 3$;
\item $(d_2, g_2) = (4, 2)$ and $n = 5$.
\end{enumerate}
Finally, we note that $C \cup_\Gamma D \to \pp^r$ as above is a limit
of immersions of smooth curves by Lemma~\ref{smoothable}.
\end{proof}

\begin{cor} Suppose that Theorem~\ref{main-3} holds for $(d, g) = (5, 1)$. Then
Theorem~\ref{main-3} holds for $(d, g)$ one of:
\[(7, 2), \quad (6, 3), \quad (9, 6), \quad (9, 7).\]
\end{cor}
\begin{proof}
For $(d, g) = (7, 2)$, we apply Lemma~\ref{lm:ind:3}, part~\ref{42}
(taking as our inductive hypothesis the truth of Theorem~\ref{main-3} for $(d', g') = (3, 0)$).

Similarly, for $(d, g) = (6, 3)$ and $(d, g) = (9, 6)$, we apply
Lemma~\ref{lm:ind:3}, part~\ref{33}
(taking as our inductive hypothesis the truth of Theorem~\ref{main-3} for $(d', g') = (3, 0)$,
and the just-established $(d', g') = (6, 3)$, respectively).

Finally, for $(d, g) = (9, 7)$, we apply Lemma~\ref{lm:ind:3}, part~\ref{46}
(taking as our inductive hypothesis the yet-to-be-established truth of Theorem~\ref{main-3}
for $(d', g') = (5, 1)$).
\end{proof}

\begin{lm} Suppose that Theorem~\ref{main-3-1} holds for $(d, g) = (7, 5)$.
Then Theorem~\ref{main-3-1} holds for $(d, g) = (8, 6)$.
\end{lm}
\begin{proof}
We simply apply Lemma~\ref{glue} with $f\colon C \cup_\Gamma D \to \pp^3$
such that $f|_C$ is a general BN-curve of degree $7$ and genus $5$,
and $f|_D$ is a line, with $C$ joined to $D$ at a set $\Gamma$ of two points.
\end{proof}

\begin{lm} \label{lm:ind:4} Suppose $N_f(-1)$ satisfies interpolation, where $f$ is a general BN-curve
of degree $d$ and genus $g$ in $\pp^4$. Then the same is true for some smooth curve of
degree $d + 6$ and genus $g + 6$, provided $d \geq 4$.
\end{lm}
\begin{proof}
We apply Lemmas~\ref{lm:hir} and~\ref{smoothable}
for $f_2$ a curve of degree $6$ and genus $3$ to $\pp^3$,
with $n = 4$.
Note that
$N_{f_2}(-1)$ satisfies interpolation by Propositions~\ref{inter} and~\ref{twist}.
\end{proof}

\begin{lm} Theorem~\ref{main-4} holds for $(d, g)$ one of:
\[(10, 6), \quad (11, 7), \quad (12, 9).\]
\end{lm}
\begin{proof}
We simply apply Lemma~\ref{lm:ind:4}
(taking as our inductive hypothesis the truth of Theorem~\ref{main-4} for
$(d', g') = (d - 6, g - 6)$).
\end{proof}

To prove the main theorems (excluding the ``conversely\ldots'' part),
it thus remains to produce five smooth curves:
\begin{enumerate}
\item For Theorem~\ref{main-3}, it suffices to find smooth curves, satisfying
the conclusions, of degrees and genera $(5, 1)$, $(7, 4)$, and $(8, 5)$.
\item For Theorem~\ref{main-3-1}, it suffices to find a smooth curve, satisfying
the conclusions, of degree $7$ and genus $5$.
\item For Theorem~\ref{main-4}, it suffices to find a smooth curve, satisfying
the conclusions, of degree $9$ and genus $5$.
\end{enumerate}

\section{Curves in Del Pezzo Surfaces \label{sec:in-surfaces}}

In this section, we analyze the normal bundles of certain curves
by specializing to immersions $f \colon C \hookrightarrow \pp^r$
of smooth curves whose images are contained in Del Pezzo
surfaces $S \subset \pp^r$ (where the Del Pezzo surface is embedded by
its complete anticanonical series).
Since $f$ will be an immersion, we shall identify $C = f(C)$ with its image,
in which case the normal bundle $N_f$ becomes the normal bundle $N_C$ of the image.
Our basic method in this section will be to use the normal bundle exact
sequence associated to $C \subset S \subset \pp^r$:
\begin{equation} \label{nb-exact}
0 \to N_{C/S} \to N_C \to N_S|_C \to 0.
\end{equation}
Since $S$ is a Del Pezzo surface, we have by adjunction an isomorphism
\begin{equation} \label{ncs}
N_{C/S} \simeq K_C \otimes K_S^\vee \simeq K_C(1).
\end{equation}

\begin{defi} \label{pic-res}
Let $S \subset \pp^r$ be a Del Pezzo surface, $k$ be an integer with $H^1(N_S(-k)) = 0$,
and $\theta \in \pic S$ be any divisor class.

Let $F$ be a general hypersurface of degree $k$.
We consider the moduli space $\mathcal{M}$ of pairs $(S', \theta')$,
with $S'$ a Del Pezzo surface containing $S \cap F$, and $\theta' \in \pic S'$.
Define $V_{\theta, k} \subseteq \pic(S \cap F)$
to be the subvariety obtained by restricting $\theta'$ to $S \cap F \subseteq S'$,
as $(S', \theta)$ varies over the component of $\mathcal{M}$ containing $(S, \theta)$.

Note that there is a unique such component, since $\mathcal{M}$ is smooth at $[(S, \theta)]$
thanks to our assumption that $H^1(N_S(-k)) = 0$.
\end{defi}

Our essential tool is given by the following lemma,
which uses the above normal bundle sequence together with the varieties
$V_{\theta, k}$ to analyze $N_C$.

\begin{lm} \label{del-pezzo}
Let $C \subset S \subset \pp^r$ be a general curve (of any fixed class)
in a general Del Pezzo surface $S \subset \pp^r$,
and $k$ be a natural number with $H^1(N_S(-k)) = 0$. Suppose that (for $F$ a general hypersurface of degree $k$):
\[\dim V_{[C], k} = \dim H^0(\oo_C(k - 1)) \quad \text{and} \quad H^1(N_S|_C(-k)) = 0,\]
and that the natural map
\[H^0(N_S(-k)) \to H^0(N_S|_C(-k))\]
is an isomorphism.
Then,
\[H^1(N_{C}(-k)) = 0.\]
\end{lm}
\begin{proof}
Twisting our earlier normal bundle exact sequence \eqref{nb-exact},
and using the isomorphism \eqref{ncs}, we obtain the exact sequence:
\[0 \to K_C(1-k) \to N_C(-k) \to N_S|_C(-k) \to 0.\]
This gives rise to a long exact sequence in cohomology:
\[\cdots \to H^0(N_C(-k)) \to H^0(N_S|_C(-k)) \to H^1(K_C(1 - k)) \to H^1(N_C(-k)) \to H^1(N_S|_C(-k)) \to \cdots.\]
Since $H^1(N_S|_C(-k)) = 0$ by assumption,
it suffices to show that the image of the natural map
$H^0(N_C(-k)) \to H^0(N_S|_C(-k))$ has codimension
\[\dim H^1(K_C(1 - k)) = \dim H^0(\oo_C(k - 1)) = \dim V_{[C], k}.\]

Because the natural map $H^0(N_S(-k)) \to H^0(N_S|_C(-k))$
is an isomorphism, we may interpret sections
of $N_S|_C(-k)$ as first-order deformations of the Del Pezzo surface $S$
fixing $S \cap F$.
So it remains to show that the space of such deformations
coming from a deformation of $C$ fixing $C \cap F$ has codimension
$\dim V_{[C], k}$.

The key point here is that deforming
$C$ on $S$ does not change its class $[C] \in \pic(S)$,
and every deformation of $S$
comes naturally with a deformation of the element $[C] \in \pic(S)$.
It thus suffices to prove that 
the space of first-order deformations of $S$ which leave invariant
the restriction $[C]|_{S \cap F} \in \pic(S \cap F)$
has codimension $\dim V_{[C], k}$.

But since the map $\mathcal{M} \to V_{[C], k}$
is smooth at $(S, [C])$, the vertical tangent space has codimension
in the full tangent space
equal to the dimension of the image.
\end{proof}

In applying Lemma~\ref{del-pezzo},
we will first consider the case where $S \subset \pp^3$ is a general cubic surface,
which is isomorphic to the blowup $\bl_\Gamma \pp^2$ of $\pp^2$ along a set
\[\Gamma = \{p_1, \ldots, p_6\} \subset \pp^2\]
of six general points. Recall that this a Del Pezzo surface,
which is to say that the embedding $\bl_\Gamma \pp^2 \simeq S \hookrightarrow \pp^3$
as a cubic surface is via the complete linear
system for the inverse of the canonical bundle:
\[-K_{\bl_\Gamma \pp^2} = 3L - E_1 - \cdots - E_6,\]
where $L$ is the class of a line in $\pp^2$ and $E_i$ is the exceptional divisor
in the blowup over $p_i$. Note that by construction,
\[N_S \simeq \oo_S(3).\]
In particular, $H^1(N_S(-1)) = H^1(N_S(-2)) = 0$ by Kodaira vanishing.

\begin{lm} \label{cubclass} Let $C \subset \bl_\Gamma \pp^2 \simeq S \subset \pp^3$ be a general curve of class either:
\begin{enumerate}
\item \label{74} $5L - 2E_1 - 2E_2 - E_3 - E_4 - E_5 - E_6$;
\item \label{85} $5L - 2E_1 - E_2 - E_3 - E_4 - E_5 - E_6$;
\item \label{86} $6L - E_1 - E_2 - 2E_3 - 2E_4 - 2E_5 - 2E_6$;
\item \label{75} $6L - E_1 - 2E_2 - 2E_3 - 2E_4 - 2E_5 - 2E_6$;
\end{enumerate}
Then $C$ is smooth and irreducible.
In the first two cases, $H^1(\oo_C(1)) = 0$.
\end{lm}
\begin{proof}
We first show the above linear series are basepoint-free.
To do this, we write each as a sum of terms which are evidently
basepoint-free:
\begin{align*}
5L - 2E_1 - 2E_2 - E_3 - E_4 - E_5 - E_6 &= (3L - E_1 - E_2 - E_3 - E_4 - E_5 - E_6) \\
&\qquad + (L - E_1) + (L - E_2) \\
5L - 2E_1 - E_2 - E_3 - E_4 - E_5 - E_6 &= (3L - E_1 - E_2 - E_3 - E_4 - E_5 - E_6) + (L - E_1) \\
6L - E_1 - E_2 - 2E_3 - 2E_4 - 2E_5 - 2E_6 &= (3L - E_1 - E_2 - E_3 - E_4 - E_5 - E_6) \\
&\qquad + L + (2L - E_3 - E_4 - E_5 - E_6) \\
6L - E_1 - 2E_2 - 2E_3 - 2E_4 - 2E_5 - 2E_6 &= (3L - E_1 - E_2 - E_3 - E_4 - E_5 - E_6) \\
&\qquad + (L - E_2) + (2L - E_3 - E_4 - E_5 - E_6).
\end{align*}

Since all our linear series are basepoint-free, the
Bertini theorem implies that $C$ is smooth. Moreover, by basepoint-freeness,
we know that $C$ does not contain any of our exceptional divisors.
We conclude that $C$ is a the proper transform in the blowup of
a curve $C_0 \subset \pp^2$. This curve satisfies:
\begin{itemize}
\item In case~\ref{74}, $C_0$ has exactly two nodes, at $p_1$ and $p_2$, and is otherwise smooth.
In particular, $C_0$ (and thus $C$) must be irreducible, since otherwise (by B\'ezout's theorem) it would have
at least $4$ nodes (where the components meet).

\item In case~\ref{85}, $C_0$ has exactly one node, at $p_1$, and is otherwise smooth.
As above, $C_0$ (and thus $C$) must be irreducible.

\item In case~\ref{86}, $C_0$ has exactly four nodes, at $\{p_3, p_4, p_5, p_6\}$, and is otherwise smooth.
As above, $C_0$ (and thus $C$) must be irreducible.

\item In case~\ref{75}, $C_0$ has exactly $5$ nodes, at $\{p_2, p_3, p_4, p_5, p_6\}$, and is otherwise smooth.
As above, $C_0$ must either be irreducible, or the union of a
line and a quintic. (Otherwise, it would have at least $8$ nodes.)
But in the second case, all $5$ nodes must be collinear,
contradicting our assumption that $\{p_2, p_3, p_4, p_5, p_6\}$ are general.
Consequently, $C_0$ (and thus $C$) must be irreducible.
\end{itemize}

We now turn to showing $H^1(\oo_C(1)) = 0$ in the first two cases.
In the first case, we note that $\Gamma$ contains $4 = \operatorname{genus}(C)$ general points $\{p_3, p_4, p_5, p_6\}$
on $C$; consequently, $E_3 + E_4 + E_5 + E_6$ --- and therefore
$\oo_C(1) = (3L - E_1 - E_2) - (E_3 + E_4 + E_5 + E_6)$ --- is a general line bundle of degree $7$,
which implies $H^1(\oo_C(1)) = 0$.
Similarly, in the second case,
we note that $\Gamma$ contains $5 = \operatorname{genus}(C)$
general points $\{p_2, p_3, p_4, p_5, p_6\}$ on $C$.
As in the first case, this implies $H^1(\oo_C(1)) = 0$, as desired.
\end{proof}

\begin{lm} \label{foo}
Let $C \subset \pp^3$ be a general BN-curve of degree and genus $(7, 4)$ or $(8, 5)$.
Then we have $H^1(N_C(-2)) = 0$.
\end{lm}
\begin{proof}
We take $C \subset S$, as constructed in Lemma~\ref{cubclass}, parts~\ref{74} and~\ref{85}
respectively.
These curves have degrees and genera $(7, 4)$ and $(8, 5)$ respectively, which can be seen by calculating the
intersection product with the hyperplane class and using adjunction.
For example, for the curve in part~\ref{74} of class
$5L - 2E_1 - 2E_2 - E_3 - E_4 - E_5 - E_6$, we calculate
\[\deg C = (5L - 2E_1 - 2E_2 - E_3 - E_4 - E_5 - E_6) \cdot (3L - E_1 - E_2 - E_3 - E_4 - E_5 - E_6) = 7,\]
and
\[\operatorname{genus} C = 1 + \frac{K_S \cdot C + C^2}{2} = 1 + \frac{-\deg C + C^2}{2} = 1 + \frac{-7 + 13}{2} = 4.\]
Because $N_S \simeq \oo_S(3)$, we have 
\[H^1(N_S|_C(-2)) = H^1(\oo_C(1)) = 0.\]
Moreover, $\oo_S(1)(-C)$ is either
$-2L + E_1 + E_2$ or $-2L + E_1$ respectively;
in either case we have $H^0(\oo_S(1)(-C)) = 0$. Consequently, the restriction map
\[H^0(\oo_S(1)) \to H^0(\oo_C(1))\]
is injective. Since
\[\dim H^0(\oo_S(1)) = 4 = \dim H^0(\oo_C(1)),\]
the above restriction map is therefore an isomorphism.
Applying
Lemma~\ref{del-pezzo}, it thus suffices to show that
\[\dim V_{[C], 2} = \dim H^0(\oo_C(1)) = 4.\]

To do this, we first observe that $[C]$ is always a linear combination $aH + bL_1 + cL_2$ of the
hyperplane class $H$, and two nonintersecting lines $L_1$ and $L_2$, such that both $b$ and $c$
are nonvanishing. Indeed:
\begin{align*}
5L - 2E_1 - 2E_2 - E_3 - E_4 - E_5 - E_6 &= 3(3L - E_1 - E_2 - E_3 - E_4 - E_5 - E_6) \\
&\quad - (2L - E_1 - E_3 - E_4 - E_5 - E_6) \\
&\quad - (2L - E_2 - E_3 - E_4 - E_5 - E_6) \\
5L - 2E_1 - E_2 - E_3 - E_4 - E_5 - E_6 &= 3(3L - E_1 - E_2 - E_3 - E_4 - E_5 - E_6) + E_1 \\
&\quad - 2(2L - E_2 - E_3 - E_4 - E_5 - E_6).
\end{align*}

Writing $F$ for a general quadric hypersurface, and $D = F \cap S$,
we observe that $\pic(D)$ is $4$-dimensional.
It is therefore sufficient to prove that for a general class $\theta \in \pic^{6a + 2b + 2c}(D)$,
there exists a smooth cubic surface $S$ containing $D$ and a pair $(L_1, L_2)$ of disjoint lines on $S$,
such that the restriction $(aH + bL_1 + cL_2)|_D = \theta$.
Since $H|_D = \oo_D(1)$ is independent of $S$ and the choice of $(L_1, L_2)$,
we may replace $\theta$ by $\theta(-a)$ and set $a = 0$.

We thus seek to show that for $b, c \neq 0$ and $\theta \in \pic^{2b + 2c}(D)$ general,
there exists a smooth cubic surface $S$ containing $D$, and a pair $(L_1, L_2)$ of disjoint lines on $S$,
with $(bL_1 + cL_2)|_D = \theta$.
Equivalently, we want to show the map
\[\{(S, E_1, E_2) : E_1, E_2 \subset S \supset D\} \mapsto \{(E_1, E_2)\},\]
from the space of smooth cubic surfaces $S$ containing $D$ with a choice
of pair of disjoint lines $(E_1, E_2)$, 
to the space of pairs of $2$-secant lines to $D$, is dominant.
For this, it suffices to check the vanishing of 
$H^1(N_S(-D -E_1 - E_2))$,
for any smooth cubic $S$ containing $D$ and disjoint lines $(E_1, E_2)$ on $S$,
in which lies the obstruction to smoothness of this map.
But $N_S(-D -E_1 - E_2) = 3L - 2E_1 - 2E_2 - E_3 - E_4 - E_5 - E_6$
has no higher cohomology by Kawamata-Viehweg vanishing.
\end{proof}

\begin{lm}
Let $C \subset \pp^3$ be a general BN-curve of degree $7$ and genus $5$.
Then we have $H^1(N_C(-1)) = 0$.
\end{lm}
\begin{proof}
We take $C \subset S$, as constructed in Lemma~\ref{cubclass}, part~\ref{75}.
Because $N_S \simeq \oo_S(3)$, we have 
\[H^1(N_S|_C(-1)) = H^1(\oo_C(2)) = 0.\]
Moreover, $\oo_S(2)(-C) \simeq \oo_S(-E_1)$ has no sections.
Consequently, the restriction map
\[H^0(\oo_S(2)) \to H^0(\oo_C(2))\]
is injective. Since
\[\dim H^0(\oo_S(2)) = 10 = \dim H^0(\oo_C(2)),\]
the above restriction map is therefore an isomorphism.
Applying
Lemma~\ref{del-pezzo}, it thus suffices to show that
\[\dim V_{[C], 1} = \dim H^0(\oo_C) = 1.\]

Writing $F$ for a general hyperplane, and $D = F \cap S$,
we observe that $\pic(D)$ is $1$-dimensional.
Since $[C] = 2H + E_1$,
it is therefore sufficient to prove that for a general class $\theta \in \pic^7(D)$,
there exists a cubic surface $S$ containing $D$ and a line $L$ on $S$,
such that the restriction $(2H + L)|_D = \theta$.
Since $H|_D = \oo_D(1)$ is independent of $S$ and the choice of $L$,
we may replace $\theta$ by $\theta(-1)$ and look instead for
$L|_D = \theta \in \pic^1(D)$.
Equivalently, we want to show the map
\[\{(S, E_1) : E_1 \subset S \supset D\} \mapsto \{(E_1, E_2)\},\]
from the space of smooth cubic surfaces $S$ containing $D$ with a choice
of line $E_1$,
to the space of $1$-secant lines to $D$, is dominant;
it suffices to check the vanishing of 
$H^1(N_S(-D-E_1))$,
for any smooth cubic $S$ containing $D$ and line $E_1$ on $S$,
in which lies the obstruction to smoothness of this map.
But $N_S(-D-E_1) = 6L - 3E_1 - 2E_2 - 2E_3 - 2E_4 - 2E_5 - 2E_6$
has no higher cohomology by Kodaira vanishing.
\end{proof}

Next, we consider the case where $S \subset \pp^4$ is the intersection
of two quadrics, which is isomorphic to the blowup $\bl_\Gamma \pp^2$
of $\pp^2$ along a set
\[\Gamma = \{p_1, \ldots, p_5\}\]
of five general points. Recall that this is a Del Pezzo surface,
which is to say that the embedding
$\bl_\Gamma \pp^2 \simeq S \hookrightarrow \pp^4$ as the intersection
of two quadrics is via the complete linear
system for the inverse of the canonical bundle:
\[-K_{\bl_\Gamma \pp^2} = 3L - E_1 - \cdots - E_5,\]
where $L$ is the class of a line in $\pp^2$ and $E_i$ is the exceptional divisor
in the blowup over $p_i$. Note that by construction,
\[N_S \simeq \oo_S(2) \oplus \oo_S(2).\]
In particular, $H^1(N_S(-1)) = 0$ by Kodaira vanishing.

\begin{lm} \label{qclass} Let $C \subset \bl_\Gamma \pp^2 \simeq S \subset \pp^4$ be a general curve of class either:
\begin{enumerate}
\item $5L - 2E_1 - E_2 - E_3 - E_4 - E_5$;
\item $6L - E_1 - 2E_2 - 2E_3 - 2E_4 - 2E_5$.
\end{enumerate}
Then $C$ is smooth and irreducible. In the first case, $H^1(\oo_C(1)) = 0$.
\end{lm}
\begin{proof}
We first show the above linear series
are basepoint-free.
To do this, we write them as a sum of terms which are evidently
basepoint-free:
\begin{align*}
5L - 2E_1 - E_2 - E_3 - E_4 - E_5 &= (3L - E_1 - E_2 - E_3 - E_4 - E_5) + (L - E_1) + L \\
6L - E_1 - 2E_2 - 2E_3 - 2E_4 - 2E_5 &= (3L - E_1 - E_2 - E_3 - E_4 - E_5) \\
&\qquad + (2L - E_2 - E_3 - E_4 - E_5) + L
\end{align*}
As in Lemma~\ref{cubclass}, we conclude that $C$ is smooth and
irreducible. In the first case, we have
$\deg \oo_C(1) = 9 > 8 = 2g - 2$, which implies
$H^1(\oo_C(1)) = 0$ as desired.
\end{proof}

\begin{lm}
Let $C \subset \pp^4$ be a general BN-curve of degree $9$ and genus $5$.
Then we have $H^1(N_C(-1)) = 0$.
\end{lm}
\begin{proof}
We take $C \subset S$, as constructed in Lemma~\ref{qclass}.
Because $N_S \simeq \oo_S(2) \oplus \oo_S(2)$, we have 
\[H^1(N_S|_C(-1)) = H^1(\oo_C(1) \oplus \oo_C(1)) = 0.\]
Moreover, $\oo_S(1)(-C) \simeq \oo_S(-2L + E_1)$ has no sections.
Consequently, the restriction map
\[H^0(\oo_S(1) \oplus \oo_S(1)) \to H^0(\oo_C(1) \oplus \oo_C(1))\]
is injective. Since
\[\dim H^0(\oo_S(1) \oplus \oo_S(1)) = 10 = \dim H^0(\oo_C(1) \oplus \oo_C(1)),\]
the above restriction map is therefore an isomorphism.
Applying
Lemma~\ref{del-pezzo}, it thus suffices to show that
\[\dim V_{[C], 1} = \dim H^0(\oo_C) = 1.\]

Writing $F$ for a general hyperplane, and $D = F \cap S$, we observe that $\pic(D)$ is $1$-dimensional.
Since $[C] = 3(3L - E_1 - E_2 - E_3 - E_4 - E_5) - 2(2L - E_1 - E_2 - E_3 - E_4 - E_5) - E_1$,
it is therefore sufficient to prove that for a general class $\theta \in \pic^9(D)$,
there exists a quartic Del Pezzo surface $S$ containing $D$, and a pair $\{L_1, L_2\}$ of
intersecting lines on $S$,
such that the restriction $(3H - 2L_1 - L_2)|_D = \theta$.
Since $H|_D = \oo_D(1)$ is independent of $S$ and the choice of $L$,
we may replace $\theta$ by $\theta^{-1}(3)$ and look instead for
$(2L_1 + L_2)|_D = \theta \in \pic^3(D)$.
For this, it suffices to show the map
\[\{(S, L_1, L_2) : L_1, L_2 \subset S \supset D\} \mapsto \{(L_1, L_2)\},\]
from the space of smooth quartic Del Pezzo surfaces $S$
containing $D$ with a choice
of pair of intersecting lines $(L_1, L_2)$, 
to the space of pairs of intersecting $1$-secant lines to $D$, is dominant.
Taking $[L_1] = E_1$ and $[L_2] = L - E_1 - E_2$,
it suffices to check the vanishing of the first cohomology of the vector bundle
$N_S(-D - E_1 - (L - E_1 - E_2))$ --- which is isomorphic to a direct
sum of two copies of the line bundle $2L - E_1 - E_3 - E_4 - E_5$ --- for
any smooth quartic Del Pezzo surface $S$ containing $D$,
in which lies the obstruction to smoothness of this map.
But $2L - E_1 - E_3 - E_4 - E_5$ has no higher cohomology by Kodaira vanishing.
\end{proof}

To prove the main theorems (excluding the ``conversely\ldots'' part),
it thus remains to produce a smooth curve $C \subset \pp^3$ of degree $5$
and genus $1$, with $H^1(N_C(-2)) = 0$.

\section{\boldmath Elliptic Curves of Degree $5$ in $\pp^3$ \label{sec:51}}

In this section, we construct an immersion $f \colon C \hookrightarrow \pp^3$
of degree~$5$ from a smooth elliptic curve,
with $H^1(N_f(-2)) = 0$.
As in the previous section,
we shall identify $C = f(C)$ with its image,
in which case the normal bundle $N_f$ becomes the normal bundle $N_C$ of the image.

Our basic method in this section will be to use the geometry of the cubic scroll $S \subset \pp^4$.
Recall that
the cubic scroll can be constructed
in two different ways:
\begin{enumerate}
\item Let $Q \subset \pp^4$ and $M \subset \pp^4$ be a plane conic,
and a line disjoint from the span of $Q$, respectively. As abstract varieties,
$Q \simeq \pp^1 \simeq M$.
Then $S$ is the ruled surface swept out by lines joining pairs of points
identified under some choice of above isomorphism.

\item Let $x \in \pp^2$ be a point, and consider the blowup $\bl_x \pp^2$
of $\pp^2$ at the point $\{x\}$. Then, $S$ is the image of $f \colon \bl_x \pp^2 \hookrightarrow \pp^4$
under the complete linear series attached to the line bundle
\[2L - E,\]
where $L$ is the class of a line in $\pp^2$, and $E$ is the exceptional divisor
in the blowup.
\end{enumerate}

To relate these two constructions, we fix a line $L \subset \pp^2$ not meeting $x$ in the second
construction, and consider the isomorphism $L \simeq \pp^1 \simeq E$
defined by sending $p \in L$ to the intersection with $E$ of the proper transform
of the line joining $p$ and $x$.
Then the $f(L)$ and $f(E)$ are $Q$ and $M$ respectively in the second construction;
the proper transforms of lines through $x$ are the lines of the ruling.

\medskip

Now take two points $p, q \in L$. Since $f(L)$ is a plane conic,
the tangent lines to $f(L)$ at $p$ and $q$ intersect; we let $y$
be their point of intersection.

From the first description of $S$, it is clear that any line through
$y$ intersects $S$ quasi-transversely --- except for the lines joining $y$ to $p$ and $q$,
each of which meets $S$ in a degree~$2$ subscheme of $f(L)$.
Write $\bar{S}$ for the image of $S$ under projection from $y$; by construction,
the projection $\pi \colon S \to \bar{S} \subseteq \pp^3$ is unramified away from $\{p, q\}$,
an immersion away from $f(L)$, and when restricted to $f(L)$ is a double cover of its image
with ramification exactly at $\{p, q\}$.
At $\{p, q\}$, the differential drops rank transversely,
with kernel the tangent
space to $f(L)$. (By ``drops rank transversely'', we mean that the section $d\pi$ of
$\hom(T_S, \pi^* T_{\pp^3})$ is transverse to the subvariety
of $\hom(T_S, \pi^* T_{\pp^3})$ of maps with less-than-maximal rank.)

If $C \subset \bl_{\{p\}} \pp^2 \simeq S$ is a curve passing through $p$ and $q$,
but transverse to $L$ at each of these points, then any line through $y$ intersects
$C$ quasi-transversely. In particular, if $C$ meets $L$ in at most one point outside of $\{p, q\}$,
the image $\bar{C}$ of $C$ under projection from $y$
is smooth. Moreover, the above analysis of $d\pi$ on $S$ implies that the natural map
\[N_{C/S} \to N_{\bar{C}/\pp^3}\]
induced by $\pi$ is fiberwise injective away from $\{p, q\}$, and has a simple
zero at both $p$ and $q$. That is, we have an exact sequence
\begin{equation} \label{51}
0 \to N_{C/S}(p + q) \to N_{\bar{C}/\pp^3} \to \mathcal{Q} \to 0,
\end{equation}
with $\mathcal{Q}$ a vector bundle.

\medskip

We now specialize to the case where $C$ is the proper transform of a plane cubic, passing through
$\{x, p, q\}$, and transverse to $L$ at $\{p, q\}$. By inspection,
$\bar{C}$ is an elliptic curve of degree $5$ in $\pp^3$; it thus suffices to show
$H^1(N_{\bar{C}/\pp^3}(-2)) = 0$.

\begin{lm} In this case, 
\begin{align*}
N_{C/S}(p + q) &\simeq \oo_C(3L - E + p + q) \\
\mathcal{Q} &\simeq \oo_C(5L - 3E - p - q).
\end{align*}
\end{lm}
\begin{proof} 
We first note that
\[N_{C/S} \simeq N_{C/\pp^2}(-E) \simeq \oo_C(3L)(-E) \quad \Rightarrow \quad N_{C/S}(p + q) \simeq \oo_C(3L - E + p + q).\]
Next, the Euler exact sequence
\[0 \to \oo_{\bar{C}} \to \oo_{\bar{C}}(1)^4 \to T_{\pp^3}|_{\bar{C}} \to 0\]
implies
\[\wedge^3 (T_{\pp^3}|_{\bar{C}}) \simeq \oo_C(4).\]
Combined with the normal bundle exact sequence
\[0 \to T_C \to T_{\pp^3}|_{\bar{C}} \to N_{\bar{C}/\pp^3} \to 0,\]
and the fact that $C$ is of genus $1$, so $T_C \simeq \oo_C$, we conclude that
\[\wedge^2(N_{\bar{C}/\pp^3}) \simeq \oo_C(4) \otimes T_C^\vee \simeq \oo_C(4) = \oo_C(4(2L - E)) = \oo_C(8L - 4E).\]
The exact sequence \eqref{51} then implies
\[\mathcal{Q} \simeq \wedge^2(N_{\bar{C}/\pp^3}) \otimes (N_{C/S}(p + q))^\vee \simeq \oo_C(8L - 4E)(-3L + E - p - q) = \oo_C(5L - 3E - p - q),\]
as desired.
\end{proof}

\noindent
Twisting by $\oo_C(-2) \simeq \oo_C(-4L + 2E)$, we obtain isomorphisms:
\begin{align*}
N_{C/S}(p + q) &\simeq \oo_C(-L + E + p + q) \\
\mathcal{Q} &\simeq \oo_C(L - E - p - q).
\end{align*}
We thus have an exact sequence
\[0 \to \oo_C(-L + E + p + q) \to N_{\bar{C}/\pp^3}(-2) \to \oo_C(L - E - p - q) \to 0.\]
Since $\oo_C(-L + E + p + q)$ and $\oo_C(L - E - p - q)$ are both general line bundles
of degree zero on a curve of genus $1$, we have
\[H^1(\oo_C(-L + E + p + q)) = H^1(\oo_C(L - E - p - q)) = 0,\]
which implies
\[H^1(N_{\bar{C}/\pp^3}(-2)) = 0.\]
This completes the proof the main theorems, except for the ``conversely\ldots'' parts.

\section{The Converses \label{sec:converses}}

In this section, we show that the intersections appearing in our main theorems
fail to be general in all listed exceptional cases.
We actually go further, describing precisely the intersection of a general BN-curve $f \colon C \to \pp^r$
in terms of the intrinsic geometry of $Q \simeq \pp^1 \times \pp^1$, $H \simeq \pp^2$,
and $H \simeq \pp^3$ respectively.

Since the general BN-curve $f \colon C \to \pp^r$ is an immersion, we can
identify $C = f(C)$ with its image as in the previous two sections, in which case
the normal bundle $N_f$ becomes the normal bundle $N_C$ of its image.

There are two basic phenomenon which occur explain the majority of our exceptional
cases: cases where $C$ is a complete intersection, and cases where $C$ lies
on a surface of low degree. The first two subsections will be devoted
to the exceptional cases that arise for these two reasons respectively.
In the final subsection, we will consider the two remaining exceptional
cases.

\subsection{Complete Intersections}

We begin by dealing with those exceptional cases which
are complete intersections.

\begin{prop}
Let $C \subset \pp^3$ be a general BN-curve of degree $4$ and genus $1$.
Then the intersection $C \cap Q$ is the intersection of two general curves
of bidegree $(2, 2)$ on $Q \simeq \pp^1 \times \pp^1$. In particular,
it is not a collection of $8$ general points.
\end{prop}
\begin{proof}
It is easy to see that $C$ is the complete intersection of two general quadrics.
Restricting these quadrics to $Q \simeq \pp^1 \times \pp^1$,
we see that $C \cap Q$ is the intersection of two general curves
of bidegree $(2, 2)$.

Since general points impose independent conditions on the $9$-dimensional
space of curves of bidegree $(2, 2)$, a general collection of $8$ points
will lie only on one curve of bidegree $(2, 2)$.
The intersections of two general curves of bidegree $(2, 2)$
is therefore not a collection of $8$ general points.
\end{proof}

\begin{prop} \label{64-to-Q}
Let $C \subset \pp^3$ be a general BN-curve of degree $6$ and genus $4$.
Then the intersection $C \cap Q$ is the intersection of two general curves
of bidegrees $(2, 2)$ and $(3,3)$ respectively on $Q \simeq \pp^1 \times \pp^1$. In particular,
it is not a collection of $12$ general points.
\end{prop}
\begin{proof}
It is easy to see that $C$ is the complete intersection of a
general quadric and cubic.
Restricting these to $Q \simeq \pp^1 \times \pp^1$,
we see that $C \cap Q$ is the intersection of two general curves
of bidegrees $(2, 2)$ and $(3,3)$ respectively.

Since general points impose independent conditions on the $9$-dimensional
space of curves of bidegree $(2, 2)$, a general collection of $12$ points
will not lie any curve of bidegree $(2,2)$, and in particular will not be
such an intersection.
\end{proof}

\begin{prop}
Let $C \subset \pp^3$ be a general BN-curve of degree $6$ and genus $4$.
Then the intersection $C \cap H$ is a general collection of $6$ points
lying on a conic. In particular,
it is not a collection of $6$ general points.
\end{prop}
\begin{proof}
As in Proposition~\ref{64-to-Q},
we see that $C \cap H$ is the intersection of general
conic and cubic curves.

In particular, $C \cap H$ lies on a conic. Conversely, any $6$ points
lying on a conic are the complete intersection of a conic and a cubic by Theorem~\ref{main-2}
(with $(d, g) = (3, 1)$).

Since general points impose independent conditions on the $6$-dimensional
space of plane conics,
a general collection of $6$ points
will not lie on a conic. We thus see our intersection
is not a collection of $6$ general points.
\end{proof}

\begin{prop}
Let $C \subset \pp^4$ be a general BN-curve of degree $8$ and genus $5$.
Then the intersection $C \cap H$ is the intersection of three general quadrics
in $H \simeq \pp^3$. In particular,
it is not a collection of $8$ general points.
\end{prop}
\begin{proof}
It is easy to see that $C$ is the complete intersection of three general quadrics.
Restricting these quadrics to $H \simeq \pp^3$,
we see that $C \cap H$ is the intersection of three general quadrics.

Since general points impose independent conditions on the $10$-dimensional
space of quadrics, a general collection of $8$ points
will lie only on only two quadrics.
The intersection of three general quadrics
is therefore not a collection of $8$ general points.
\end{proof}

\subsection{Curves on Surfaces}

Next, we analyze those cases
which are exceptional because $C$ lies on a surface $S$
of small degree. To show the intersection is general subject to
the constraint imposed by $C \subset S$, it will be useful to have the following lemma:

\begin{lm} \label{pic-res-enough}
Let $D$ be an irreducible curve of genus $g$ on a surface $S$, and $p_1, p_2, \ldots, p_n$
be a collection of $n$ distinct points on $D$. Suppose that $n \geq g$, and that
$p_1, p_2, \ldots, p_g$ are general.
Let $\theta \in \pic(S)$, with $\theta|_D \sim p_1 + p_2 + \cdots + p_n$. Suppose that
\[\dim H^0(\theta) - \dim H^0(\theta(-D)) \geq n - g + 1.\]
Then some curve $C \subset S$ of class $\theta$ meets $D$ transversely at $p_1, p_2, \ldots, p_n$.
\end{lm}
\begin{proof}
Since $p_1, p_2, \ldots, p_g$ are general, and $\theta|_D = p_1 + p_2 + \cdots + p_n$,
it suffices to show there is a curve of class $\theta$ meeting $D$ dimensionally-transversely
and passing through $p_{g + 1}, p_{g + 2}, \ldots, p_n$; the remaining $g$ points
of intersection are then forced to be $p_1, p_2, \ldots, p_g$.

For this, we note there is a $\dim H^0(\theta) - (n - g) > \dim H^0(\theta(-D))$ dimensional
space of sections of $\theta$ which vanish at $p_{g + 1}, \ldots, p_n$.
In particular, there is some section which does not vanish along $D$.
Its zero locus then gives the required curve $C$. (The curve $C$ meets $D$ dimensionally-transversely,
because $C$ does not contain $D$ and $D$ is irreducible.)
\end{proof}

\begin{prop}
Let $C \subset \pp^3$ be a general BN-curve of degree $5$ and genus $2$.
Then the intersection $C \cap Q$ is a collection of $10$ general points
lying on a curve of bidegree $(2, 2)$ on $Q \simeq \pp^1 \times \pp^1$. In particular,
it is not a collection of $10$ general points.
\end{prop}
\begin{proof}
Since $\dim H^0(\oo_C(2)) = 9$ and $\dim H^0(\oo_{\pp^3}(2)) = 10$,
we conclude that $C$ lies on a quadric.
Restricting to $Q$, we see that $C \cap Q$ lies on a curve
of bidegree $(2,2)$.

Conversely, given $10$ points $p_1, p_2, \ldots, p_{10}$ lying on a curve $D$ of bidegree $(2, 2)$,
we may first find a pair of points $\{x, y\} \subset D$ so that
$x + y + 2H \sim p_1 + \cdots + p_{10}$. We then claim there is a smooth quadric containing
$D$ and the general $2$-secant line $\overline{xy}$ to $D$.
Equivalently, we want to show the map
\[\{(S, L) : L \subset S \supset D\} \mapsto \{L\},\]
from the space of smooth quadric surfaces $S$ containing $D$ with a choice
of line $L$,
to the space of $2$-secant lines to $D$, is dominant;
it suffices to check the vanishing of 
$H^1(N_S(-D-L))$,
for any smooth quadric $S$ containing $D$ and line $L$ on $S$,
in which lies the obstruction to smoothness of this map.
But $N_S(-D-L) = \oo_S(0, -1)$
has no higher cohomology by Kodaira vanishing.

Writing $L \in \pic(S)$ for the class of the line $\overline{xy}$,
we see that $(L + 2H)|_D \sim p_1 + \cdots + p_{10}$ as divisor classes.
Applying Lemma~\ref{pic-res-enough}, and noting that
$\dim H^0(\oo_{S}(2H + L)) = 12$ while 
$\dim H^0(\oo_{S}(L)) = 2$, there is a curve $C$ of class
$2H + L$ meeting $D$ transversely at $p_1, \ldots, p_{10}$.
Since $\oo_{S}(2H + L)$ is very ample by inspection, $C$
is smooth (for $p_1, \ldots, p_{10}$ general). By results of \cite{keem},
this implies $C$ is a BN-curve.

Since general points impose independent conditions on the $9$-dimensional
space of curves of bidegree $(2, 2)$, a general collection of $10$ points
does not lie on a curve of bidegree $(2, 2)$.
A collection of $10$ general points on a general curve of bidegree $(2,2)$
is therefore not a collection of $10$ general points.
\end{proof}

\begin{prop}
Let $C \subset \pp^3$ be a general BN-curve of degree $7$ and genus $5$.
Then the intersection $C \cap Q$ is a collection of $14$
points lying on a curve $D \subset Q \simeq \pp^1 \times \pp^1$,
which is general subject to the following conditions:
\begin{enumerate}
\item The curve $D$ is of bidegree $(3, 3)$.
\item The divisor $C \cap Q - 2H$ on $D$ (where $H$ is the hyperplane class)
is effective.
\end{enumerate}
In particular, it is not a collection of $14$ general points.
\end{prop}
\begin{proof}
First we claim the general such curve $C$ lies on a smooth cubic surface $S$ with class
$2H + E_1 = 6L - E_1 - 2E_2 - 2E_3 - 2E_4 - 2E_5 - 2E_6$.
Indeed, by Lemma~\ref{cubclass} part~\ref{75}, a general curve of this class is smooth and irreducible;
such a curve has degree~$7$ and genus~$5$, and in particular is a BN-curve by results of \cite{keem}.
It remains to see there are no obstructions to lifting a deformation
of $C$ to a deformation of the pair $(S, C)$,
i.e.\ that $H^1(N_S(-C)) = 0$. But $N_S(-C) = 3L - 2E_1 - E_2 - E_3 - E_4 - E_5 - E_6$,
which has no higher cohomology by Kodaira vanishing.

Thus, $C \cap Q - 2H$ is the restriction to $D$
of the class of a line on $S$; in particular, $C \cap Q - 2H$
is an effective divisor on $D$.

Conversely, suppose that $p_1, p_2, \ldots, p_{14}$
are a general collection of $14$ points lying on a curve $D$ of bidegree $(3,3)$
with $p_1 + \cdots + p_{14} - 2H \sim x + y$ effective.
We then claim there is a smooth cubic containing
$D$ and the general $2$-secant line $\overline{xy}$ to $D$.
Equivalently, we want to show the map
\[\{(S, L) : L \subset S \supset D\} \mapsto \{L\},\]
from the space of smooth cubic surfaces $S$ containing $D$ with a choice
of line $L$,
to the space of $2$-secant lines to $D$, is dominant;
for this it suffices to check the vanishing of 
$H^1(N_S(-D-L))$.
But $N_S(-D-L) = 3L - 2E_1 - E_2 - E_3 - E_4 - E_5 - E_6$,
which has no higher cohomology by Kodaira vanishing.

Choosing an isomorphism $S \simeq \bl_\Gamma \pp^2$ where $\Gamma = \{q_1, q_2, \ldots, q_6\}$,
so that the line $\overline{xy} = E_1$ is the exceptional
divisor over $q_1$,
we now look for a curve $C \subset S$ of class
\[[C] = 6L - E_1 - 2E_2 - 2E_3 - 2E_4 - 2E_5 - 2E_6.\]

Again by Lemma~\ref{cubclass}, the general such curve is smooth and irreducible;
such a curve has degree~$7$ and genus~$5$, and in particular is a BN-curve by results of \cite{keem}.
Note that
\[\dim H^0(\oo_S(6L - E_1 - 2E_2 - 2E_3 - 2E_4 - 2E_5 - 2E_6)) = 12 \quad \text{and} \quad \dim H^0(\oo_S(E_1)) = 1.\]
Applying Lemma~\ref{pic-res-enough},
we conclude that some curve of our given class meets $D$ transversely
at $p_1, p_2, \ldots, p_{14}$, as desired.

It remains to see from this description that
$C \cap Q$ is not a general collection of $14$ points.
For this, first note that there is a $15$-dimensional space
of such curves $D$ (as $\dim H^0(\oo_Q(3,3)) = 16$).
On each each curve, there is a $2$-dimensional family of effective
divisors $\Delta$; and for fixed $\Delta$, a $10$-dimensional family of divisors
linearly equivalent to $2H + \Delta$ (because $\dim H^0(\oo_D(2H + \Delta)) = 11$
by Riemann-Roch). Putting this together,
there is an (at most) $15 + 2 + 10 = 27$-dimensional family of such collections
of points.
But $\sym^{14}(Q)$ has dimension $28$. In particular, collections of such 
points cannot be general.
\end{proof}

\begin{prop}
Let $C \subset \pp^3$ be a general BN-curve of degree $8$ and genus $6$.
Then the intersection $C \cap Q$ is a general collection of $16$ points
on a curve of bidegree $(3,3)$ on $Q \simeq \pp^1 \times \pp^1$. In particular,
it is not a collection of $16$ general points.
\end{prop}
\begin{proof}
Since $\dim H^0(\oo_C(3)) = 19$ and $\dim H^0(\oo_{\pp^3}(3)) = 20$,
we conclude that $C$ lies a cubic surface. Restricting this cubic
to $Q$, we see that $C \cap Q$ lies on a curve of bidegree $(3,3)$.

Conversely, take a general collection $p_1, \ldots, p_{16}$ of $16$ points on a curve
$D$ of bidegree $(3,3)$. The divisor $p_1 + \cdots + p_{16} - 2H$ is of degree $4$
on a curve $D$ of genus $4$; it is therefore effective, say
\[p_1 + \cdots + p_{16} - 2H \sim x + y + z + w.\]
We then claim there is a smooth cubic containing
$D$ and the general $2$-secant lines $\overline{xy}$ and $\overline{zw}$ to $D$.
Equivalently, we want to show the map
\[\{(S, E_1, E_2) : E_1, E_2 \subset S \supset D\} \mapsto \{(E_1, E_2)\},\]
from the space of smooth cubic surfaces $S$ containing $D$ with a choice
of pair of disjoint lines $(E_1, E_2)$,
to the space of pairs of $2$-secant lines to $D$, is dominant;
for this it suffices to check the vanishing of
$H^1(N_S(-D-E_1 - E_2))$.
But $N_S(-D-E_1 - E_2) = 3L - 2E_1 - 2E_2 - E_3 - E_4 - E_5 - E_6$,
which has no higher cohomology by Kawamata-Viehweg vanishing.

We now look for a curve $C \subset S$ of class
\[[C] = 6L - E_1 - E_2 - 2E_3 - 2E_4 - 2E_5 - 2E_6,\]
which is of degree $8$ and genus $6$.
By Lemma~\ref{cubclass}, we conclude that $C$ is smooth and irreducible;
by results of \cite{keem}, this implies the general curve of this class is a BN-curve.
Note that
\[\dim H^0(\oo_S(6L - E_1 - E_2 - 2E_3 - 2E_4 - 2E_5 - 2E_6)) = 14 \quad \text{and} \quad \dim H^0(\oo_S(E_1 + E_2)) = 1.\]
Applying Lemma~\ref{pic-res-enough},
we conclude that some curve of our given class meets $D$ transversely
at $p_1, p_2, \ldots, p_{16}$, as desired.

Since general points impose independent conditions on the $16$-dimensional
space of curves of bidegree $(3, 3)$, a general collection of $16$ points
will not lie any curve of bidegree $(3,3)$. Our collection of points
is therefore not general.
\end{proof}

\begin{prop}
Let $C \subset \pp^4$ be a general BN-curve of degree $9$ and genus $6$.
Then the intersection $C \cap H$ is a general collection of $9$ points
on an elliptic normal curve
in $H \simeq \pp^3$. In particular,
it is not a collection of $9$ general points.
\end{prop}
\begin{proof}
Since $\dim H^0(\oo_C(2)) = 13$ and $\dim H^0(\oo_{\pp^4}(2)) = 15$,
we conclude that $C$ lies on the intersection of two quadrics.
Restricting these quadrics to $H \simeq \pp^3$,
we see that $C \cap H$ lies on the intersection of two quadrics,
which is an elliptic normal curve.

Conversely, let $p_1, p_2, \ldots, p_9$ be a collection of $9$ points
lying on an elliptic normal curve $D \subset \pp^3$.
Since $D$ is an elliptic curve, there exists (a unique) $x \in D$
with
\[\oo_D(p_1 + \cdots + p_9)(-2) \simeq \oo_D(x).\]
Let $M$ be a general line through $x$.
We then claim there is a quartic Del Pezzo surface containing
$D$ and the general $1$-secant line $M$.
Equivalently, we want to show the map
\[\{(S, E_1) : E_1 \subset S \supset D\} \mapsto \{E_1\},\]
from the space of smooth Del Pezzo surfaces $S$ containing $D$ with a choice
of line $E_1$,
to the space of $1$-secant lines to $D$, is dominant;
for this it suffices to check the vanishing of
$H^1(N_S(-D-E_1))$.
But $N_S(-D-E_1)$ is a direct sum of two copies of the line bundle
$3L - 2E_1 - E_2 - E_3 - E_4 - E_5$,
which has no higher cohomology by Kodaira vanishing.

We now consider curves $C \subset S$ of class
\[[C] = 6L - E_1 - 2E_2 - 2E_3 - 2E_4 - 2E_5,\]
which are of degree $9$ and genus $6$.
By Lemma~\ref{qclass}, we conclude that $C$ is smooth and irreducible;
by results of \cite{iliev}, this implies the general curve of this class is a BN-curve.
Note that
\[\dim H^0(\oo_S(6L - E_1 - 2E_2 - 2E_3 - 2E_4 - 2E_5)) = 15 \quad \text{and} \quad \dim H^0(\oo_S(3L - E_2 - E_3 - E_4 - E_5)) = 6.\]
Applying Lemma~\ref{pic-res-enough},
we conclude that some curve of our given class meets $D$ transversely
at $p_1, p_2, \ldots, p_9$, as desired.

By Corollary~1.4 of \cite{firstpaper}, there does not exist an elliptic
normal curve in $\pp^3$ passing through $9$ general points.
\end{proof}

\subsection{The Final Two Exceptional Cases}

We have exactly two remaining exceptional cases: The intersection
of a general BN-curve of degree $6$ and genus $2$ in $\pp^3$ with a quadric,
and the intersection of a general BN-curve of degree $10$ and genus $7$ in $\pp^4$
with a hyperplane. We will show in the first case that the intersection fails
to be general since $C$ is the projection of a curve $\tilde{C} \subset \pp^4$,
where $\tilde{C}$ lies on a surface of small degree (a cubic scroll).
In the second case, the intersection fails to be general since $C$
is contained in a quadric hypersurface.

\begin{prop}
Let $C \subset \pp^3$ be a general BN-curve of degree $6$ and genus $2$.
Then the intersection $C \cap Q$ is a collection of $12$ points
lying on a curve $D \subset Q \simeq \pp^1 \times \pp^1$, which is general subject
to the following conditions:
\begin{enumerate}
\item The curve $D$ is of bidegree $(3, 3)$ (and so is in particular of arithmetic genus $4$).
\item The curve $D$ has two nodes (and so is in particular of geometric genus $2$).
\item The divisors $\oo_D(2,2)$ and $C \cap D$ are linearly equivalent
when pulled back to the normalization of $D$.
\end{enumerate}
In particular, it is not a collection of $12$ general points.
\end{prop}
\begin{proof}
We first observe that $\dim H^0(\oo_C(1)) = 5$, so $C$ is the projection from a point $p \in \pp^4$
of a curve $\tilde{C} \subset \pp^4$ of degree $6$ and genus $2$.
Write $\pi \colon \pp^4 \dashedrightarrow \pp^3$ for the map of projection
from $p$, and define the quadric hypersurface $\tilde{Q} = \pi^{-1}(Q)$.

Let $S \subset \pp^4$ be the surface swept out by joining pairs
of points on $\tilde{C}$ conjugate under the hyperelliptic involution.
By Corollary~13.3 of \cite{firstpaper}, $S$ is a cubic surface;
in particular, since $S$ has a ruling, $S$ is a cubic scroll.
Write $H$ for the hyperplane section on $S$, and $F$ for the class
of a line of the ruling.

Th curve $\tilde{D} = \tilde{Q} \cap S$
(which for $C$ general is smooth by Kleiman transversality), is of degree $6$ and genus $2$.
By construction, the intersection $C \cap Q$ lies on $D = \pi(\tilde{D})$. Since $D = \pi(S) \cap Q$,
it is evidently a curve of bidegree $(3, 3)$ on $Q \simeq \pp^1 \times \pp^1$.
Moreover, since $\tilde{D}$ has genus $2$, the geometric genus of $D$ is $2$.
In particular, $D$ has two nodes.

Next, we note that on $S$, the curve $\tilde{C}$ has class $2H$. Indeed, if $[\tilde{C}] = a \cdot H + b \cdot F$,
then $a = \tilde{C} \cdot F = 2$ and $3a + b = \tilde{C} \cdot H = 6$; solving for $a$ and $b$, we obtain
$a = 2$ and $b = 0$.
Consequently, $\tilde{C} \cap \tilde{D}$ has class $2H$ on $\tilde{D}$.
Or equivalently, $C \cap D = \pi(\tilde{C} \cap \tilde{D})$ has class
equal to $\oo_D(2) = \oo_D(2,2)$ when pulled back to the normalization.

Conversely, take $12$ points on $D$ satisfying our assumptions. Write
$\tilde{D}$ for the normalization of $D$, and $p_1, p_2, \ldots, p_{12}$
for the preimages of our points in $\tilde{D}$.
We begin by noting that $\dim H^0(\oo_{\tilde{D}}(1)) = 5$,
so $D$ is the projection from a point $p \in \pp^4$
of $\tilde{D} \subset \pp^4$ of degree $6$ and genus $2$.
As before, write $\pi \colon \pp^4 \dashedrightarrow \pp^3$ for the map of projection
from $p$, and define the quadric hypersurface $\tilde{Q} = \pi^{-1}(Q)$.

Again, we let $S \subset \pp^4$ be the surface swept out by joining pairs
of points on $\tilde{D}$ conjugate under the hyperelliptic involution.
As before, $S$ is a cubic scroll;
write $H$ for the hyperplane section on $S$, and $F$ for the class
of a line of the ruling.
Note that $\tilde{D} \subseteq \tilde{Q} \cap S$; and since both
sides are curves of degree $6$, we have $\tilde{D} = \tilde{Q} \cap S$.

It now suffices to find a curve $\tilde{C} \subset S$ of class $2H$,
meeting $\tilde{D}$ transversely 
in $p_1, \ldots, p_{12}$. 
For this, note that
\[\dim H^0(\oo_S(2H)) = 12 \quad \text{and} \quad \dim H^0(\oo_S) = 1.\]
Applying Lemma~\ref{pic-res-enough} yields the desired conclusion.

It remains to see from this description that
$C \cap Q$ is not a general collection of $12$ points.
For this, we first note that such a curve $D \subset \pp^1 \times \pp^1$
is the same as specifying an abstract curve of genus $2$, two lines bundles
of degree $3$ (corresponding to the pullbacks of $\oo_{\pp^1}(1)$ from each factor),
and a basis-up-to-scaling for their space of sections (giving us two maps $D \to \pp^1$).
Since there is a $3$-dimensional moduli space of abstract curves $D$ of genus $2$,
and $\dim \pic^3(D) = 2$, and there is a $3$-dimensional family of bases-up-to-scaling
of a $2$-dimensional vector space, the dimension of the space
of such curves $D$ is $3 + 2 + 2 + 3 + 3 = 13$.
Our condition $p_1 + \cdots + p_{12} \sim 2H$ then implies
collections of such points on a fixed $D$ are in bijection with
elements of $\pp \oo_D(2H) \simeq \pp^{10}$. Putting this together,
there is an (at most) $13 + 10 = 23$ dimensional family of such collections of points.
But $\sym^{12}(Q)$ has dimension $24$. In particular, collections of such 
points cannot be general.
\end{proof}

\begin{prop}
Let $C \subset \pp^4$ be a general BN-curve of degree $10$ and genus $7$.
Then the intersection $C \cap H$ is a general collection of $10$ points
on a quadric in $H \simeq \pp^3$. In particular,
it is not a collection of $10$ general points.
\end{prop}
\begin{proof}
Since $\dim H^0(\oo_C(2)) = 14$ and $\dim H^0(\oo_{\pp^4}(2)) = 15$,
we conclude that $C$ lies on a quadric.
Restricting this quadric to $H \simeq \pp^3$,
we see that $C \cap H$ lies on a quadric.

For the converse, we take general points $p_1, \ldots, p_{10}$
lying on a general (thus smooth) quadric~$Q$.
Since $\dim H^0(\oo_Q(3,3)) = 16$, we may find a curve $D \subset Q$
of type $(3,3)$ passing through $p_1, \ldots, p_{10}$.
As divisor classes on $D$, suppose that
\[p_1 + p_2 + \cdots + p_{10} - H \sim x + y + z + w.\]
We now pick a general (quartic) rational normal curve $R \subset \pp^4$
whose hyperplane section is $\{x, y, z, w\}$.

We then claim there is a smooth sextic K3 surface $S \subset \pp^4$
containing $D$ and the general $2$-secant lines $\overline{xy}$ and $\overline{zw}$ to $D$.
Equivalently, we want to show the map
\[\{(S, R) : R \subset S\} \mapsto \{(R, D)\},\]
from the space of smooth sextic K3 surfaces $S$,
to the space of pairs $(R, D)$ where $R$ is a rational normal curve
meeting the canonical curve $D = S \cap H$ in four points, is dominant;
for this it suffices to check the vanishing of
$H^1(N_S(-H-R))$ at any smooth sextic K3 containing a rational normal curve $R$
(where $H = [D]$ is the hyperplane class on $S$).
We first note that a sextic K3 surface $S$ containing a rational normal curve $R$
exists, by Theorem~1.1 of~\cite{knutsen}.
On this K3 surface, our vector bundle $N_S(-H-R)$ is the direct sum of the line bundles $H - R$ and $2H - R$;
consequently, it suffices to show $H^1(\oo_S(n)(-R)) = 0$ for $n \geq 1$.
For this we use the exact sequence
\[0 \to \oo_S(n)(-R) \to \oo_S(n) \to \oo_S(n)|_R = \oo_R(n) \to 0,\]
and note that $H^1(\oo_S(n)) = 0$ by Kodaira vanishing,
while $H^0(\oo_S(n)) \to H^0(\oo_R(n))$ is surjective since $R$ is projectively normal.
This shows the existence of the desired K3 surface $S$ containing
$D$ and the general $4$-secant rational normal curve $R$.

Next, we claim that the linear series $H + R$ on $S$ is basepoint-free.
To see this, we first note that $H$ is basepoint free, so any basepoints
must lie on the curve $R$. Now the short exact sequence of sheaves
\[0 \to \oo_S(H) \to \oo_S(H + R) \to \oo_S(H + R)|_R \to 0\]
gives a long exact sequence in cohomology
\[\cdots \to H^0(\oo_S(H + R)) \to H^0(\oo_S(H + R)|_R) \to H^1(\oo_S(H)) \to \cdots.\]

Since the complete linear series
attached to $\oo_S(H + R)|_R \simeq \oo_{\pp^1}(2)$ is basepoint-free,
it suffices to show that 
$H^0(\oo_S(H + R)) \to H^0(\oo_S(H + R)|_R)$ is surjective. For this,
it suffices to note that $H^1(\oo_S(H)) = 0$ by Kodaira vanishing.

Thus, $H + R$ is basepoint-free. In particular, the Bertini
theorem implies the general curve of class $H + R$ is smooth.
Such a curve is of degree~$10$ and genus~$7$;
in particular it is a BN-curve by results 
of \cite{iliev}.
So it suffices to find a curve of class $H + R$ on $S$
passing through $p_1, p_2, \ldots, p_{10}$.
By construction, as divisors on $D$, we have
\[p_1 + p_2 + \cdots + p_{10} \sim H + R.\]
By Lemma~\ref{pic-res-enough}, it suffices to show
$\dim H^0(\oo_S(H + R)) = 8$ and $\dim H^0(\oo_S(R)) = 1$.

More generally,
for any smooth curve $X \subset S$
of genus $g$,
we claim $\dim H^0(\oo_S(X)) = 1 + g$. To see this, we use the exact sequence
\[0 \to \oo_S \to \oo_S(X) \to \oo_S(X)|_X \to 0,\]
which gives rise to a long exact sequence in cohomology
\[0 \to H^0(\oo_S) \to H^0(\oo_S(X)) \to H^0(\oo_S(X)|_X) \to H^1(\oo_S) \to \cdots.\]
Because $H^1(\oo_S) = 0$, we thus have
\begin{align*}
\dim H^0(\oo_S(X)) &= \dim H^0(\oo_S(X)|_X) + \dim H^0(\oo_S) \\
&= \dim H^0(K_S(X)|_X) + 1 \\
&= \dim H^0(K_X) + 1 \\
&= g + 1.
\end{align*}
In particular, $\dim H^0(\oo_S(H + R)) = 8$ and $\dim H^0(\oo_S(R)) = 1$,
as desired.

Since general points impose independent conditions on the $10$-dimensional
space of quadrics, a general collection of $10$ points
will not lie on a quadric. In particular, our hyperplane
section here is not a general collection of $10$ points.
\end{proof}

\bibliography{iqbib}{}
\bibliographystyle{plain}

\end{document}